\documentclass[UTF-8,reqno]{amsart}
\usepackage{algorithm}
\usepackage{algorithmic}
\usepackage{amsmath}

\usepackage{enumerate, bbm, framed}

\usepackage{amsmath, graphicx, float} 

\usepackage{amssymb,url,color, booktabs}
\usepackage{tikz}
\usepackage{mathrsfs}
\usepackage{dutchcal}
\usepackage{threeparttable}
\usepackage{color}
\usepackage[colorlinks=true]{hyperref}
\hypersetup{
    linkcolor=blue,          
    citecolor=red,        
    filecolor=blue,      
    urlcolor=cyan
}

\usepackage{caption}
\usepackage{lipsum}

\usepackage{pgfplots}
\pgfplotsset{compat=newest}
\usepackage{color}
\usepackage{ulem}

 \usepackage{scalerel} 

\definecolor{MyDarkBlue}{cmyk}{0.8,0.3,0.8,0.4}
\definecolor{yellow}{rgb}{0.99,0.99,0.70}
\definecolor{white}{rgb}{1.0,1.0,1.0}
\definecolor{black}{rgb}{0.00,0.00,0.00}

\numberwithin{equation}{section}

\newcommand{\be}{\begin{eqnarray}}
\newcommand{\ee}{\end{eqnarray}}
\newcommand{\ce}{\begin{eqnarray*}}
\newcommand{\de}{\end{eqnarray*}}
\newtheorem{theorem}{Theorem}[section]
\newtheorem{lemma}[theorem]{Lemma}
\newtheorem{remark}[theorem]{Remark}
\newtheorem{definition}[theorem]{Definition}
\newtheorem{proposition}[theorem]{Proposition}
\newtheorem{Examples}[theorem]{Example}
\newtheorem{corollary}[theorem]{Corollary}

\usepackage[nobysame]{amsrefs}
\BibSpec{article}{%
+{}{\PrintAuthors} {author}
+{,}{ \textrm} {title}
+{.}{ \textit} {journal}
+{,}{ \textbf} {volume}
+{}{ \parenthesize} {date}
+{,}{ } {pages}
+{.}{ arXiv:} {eprint}
+{.}{} {transition}
}
\BibSpec{book}{%
+{}{\PrintAuthors} {author}
+{,}{ \textit} {title}
+{.}{ \textrm} {series} 
+{,}{ Vol.} {volume} 
+{.}{ } {publisher}
+{,}{ } {date}
+{.}{} {transition}
}

\def\var{{\mathrm{var}}}

\def\eps{\epsilon}

\def\e{\mathrm{e}}

\def\p{\partial}

\def\[{{\Big[}}
\def\]{{\Big]}}
\def\<{{\langle}}
\def\>{{\rangle}}
\def\({{\Big(}}
\def\){{\Big)}}

\def\by{{\mathbf{y}}}
\def\bx{{\mathbf{x}}}

\def\dif{{\mathord{{\rm d}}}}

\def\no{\nonumber}
\def\={&\!\!=\!\!&}

\def\cB{{\mathcal B}}
\def\cC{{\mathcal C}}
\def\cD{{\mathcal D}}
\def\cE{{\mathcal E}}

\def\cH{{\mathcal H}}

\def\cW{{\mathcal W}}

\def\mE{{\mathbb E}}

\def\mI{{\mathbb I}}

\def\mN{{\mathbb N}}

\def\mP{{\mathbb P}}

\def\mR{{\mathbb R}}

\def\1{{\mathbf{1}}}

\def\sI{{\mathscr I}}

\def\sL{{\mathscr L}}

\def\sN{{\mathscr N}}

\def\sV{{\mathscr V}}

\def\fm{{\mathfrak m}}

\def\geq{\geqslant}
\def\leq{\leqslant}

\def\div{\mathord{{\rm div}}}

\def\var{{\mathrm{var}}}

\def\e{\mathrm{e}}

\def\p{\partial}

\def\[{{\Big[}}
\def\]{{\Big]}}
\def\<{{\langle}}
\def\>{{\rangle}}
\def\({{\Big(}}
\def\){{\Big)}}

\def\by{{\boldsymbol{y}}}
\def\bx{{\mathbf{x}}}

\def\dif{{\mathord{{\rm d}}}}

\def\no{\nonumber}
\def\={&\!\!=\!\!&}
\def\bt{\begin{theorem}}
\def\et{\end{theorem}}
\def\bl{\begin{lemma}}
\def\el{\end{lemma}}
\def\br{\begin{remark}}
\def\er{\end{remark}}
\def\bx{\begin{Examples}}
\def\ex{\end{Examples}}
\def\bd{\begin{definition}}
\def\ed{\end{definition}}
\def\bp{\begin{proposition}}
\def\ep{\end{proposition}}
\def\bc{\begin{corollary}}
\def\ec{\end{corollary}}

\def\geq{\geqslant}
\def\leq{\leqslant}

\def\div{\mathord{{\rm div}}}

\def\bxi{{\boldsymbol{\xi}}}

\def\<{\langle} \def\>{\rangle}

\def\bpf{\begin{proof}}
\def\epf{\end{proof}}

\allowdisplaybreaks

\begin{document}

\title[Transport maps in in diffusion models and sampling]
{Stochastic Transport Maps in Diffusion Models and Sampling}
\author{Xicheng Zhang}

\address{Xicheng Zhang:
School of Mathematics and Statistics, Beijing Institute of Technology, Beijing 100081, China\\
Faculty of Computational Mathematics and Cybernetics, Shenzhen MSU-BIT University, 518172 Shenzhen, China\\
		Email: xczhang.math@bit.edu.cn
 }


\thanks{
This work is supported by National Key R\&D program of China (No. 2023YFA1010103) and NNSFC grant of China (No. 12131019)  and the DFG through the CRC 1283 ``Taming uncertainty and profiting from randomness and low regularity in analysis, stochastics and their applications''. }

\begin{abstract}
In this work, we present a theoretical and computational framework for constructing stochastic transport maps between probability distributions using diffusion processes. We begin by proving that the time-marginal distribution of the sum of two independent diffusion processes satisfies a Fokker-Planck equation. Building on this result and applying Ambrosio-Figalli-Trevisan's superposition principle, we establish the existence and uniqueness of solutions to the associated stochastic differential equation (SDE). Leveraging these theoretical foundations, we develop a method to construct (stochastic) transport maps between arbitrary probability distributions using dynamical ordinary differential equations (ODEs) and SDEs. Furthermore, we introduce a unified framework that generalizes and extends a broad class of diffusion-based generative models and sampling techniques. Finally, we analyze the convergence properties of particle approximations for the SDEs underlying our framework, providing theoretical guarantees for their practical implementation. This work bridges theoretical insights with practical applications, offering new tools for generative modeling and sampling in high-dimensional spaces.


\end{abstract}

\maketitle
\setcounter{tocdepth}{2}

\tableofcontents

\section{Introduction}

\subsection{Transport Maps}
Given two probability distributions $\mu$ and $\nu$ defined on $\mR^d$, a transport map 
is a function $T: \mR^d \to \mR^d$ that ``pushes forward'' the mass from $\mu$ to $\nu$. 
Mathematically, this is expressed as:
$$
T_\# \mu = \nu \quad \Leftrightarrow \quad \nu(A) = \mu(T^{-1}(A)), \quad \forall A \in \cB(\mR^d),
$$
where $\cB(\mR^d)$ denotes the Borel sets of $\mR^d$. In simpler terms, $T$ redistributes the mass of $\mu$ to match the distribution of $\nu$. Transport maps provide a powerful framework for sampling from complex probability distributions, which is particularly useful in applications such as Bayesian inference, generative modeling, and Monte Carlo methods \cite{Villani2009, Santambrogio2015}.

\smallskip\noindent {\sf Sampling via Transport Maps.}
If a transport map $T$ is known, sampling from $\nu$ can be achieved efficiently through the following steps:
\begin{enumerate}
    \item Draw a sample $\xi \sim \mu$ from a simple distribution (e.g., a Gaussian).
    \item Map it through $T$, i.e., $\eta = T(\xi)$, to obtain a sample from the target distribution $\nu$.
\end{enumerate}

This approach is especially valuable when direct sampling from $\nu$ is difficult, but evaluating $T(x)$ and sampling from $\mu$ are straightforward. For example, in Bayesian statistics, posterior distributions are often complex and challenging to sample from directly. A transport map can be learned to transform a simple prior distribution (e.g., Gaussian) into the posterior distribution. Once the map is learned, posterior samples can be generated efficiently \cite{Marzouk2016, Spantini2018}.

\smallskip\noindent {\sf Challenges in Constructing Transport Maps.}
While transport maps offer a promising approach to sampling, constructing them is often non-trivial. 
Indeed, the existence of a transport map is not guaranteed in general. For instance:
\begin{itemize}
    \item If $\mu$ is a Dirac measure and $\nu$ is not, no such map exists.
    \item Even when $\mu$ and $\nu$ are absolutely continuous, finding an explicit or computational 
    transport map can be challenging.
\end{itemize}
However, under certain conditions, the existence of transport maps can be established. For example:
\begin{itemize}
    \item Moser's mapping and the Knothe--Rosenblatt rearrangement provide constructive methods to show the existence of transport maps for absolutely continuous measures \cite{Villani2009, Santambrogio2015}.
    \item Solving the Monge optimal transport problem yields a transport map that minimizes the total cost of moving probability mass. This not only provides a deterministic map but also ensures efficiency in sampling \cite{Brenier1991, Peyre2019}.
\end{itemize}
\smallskip\noindent {\sf Regularization and Approximation.}
In practice, exact solutions to the Monge problem can be computationally expensive, especially in high dimensions. To address this, entropic regularization (e.g., Sinkhorn distances) is often used to find smooth approximations of transport maps. These regularized maps are computationally tractable and provide a practical way to approximate the target distribution \cite{Cuturi2013, Genevay2018}.

In the context of deep learning, normalizing flows parameterize transport maps using neural networks. These maps push a simple distribution (e.g., Gaussian) onto a complex target distribution, enabling efficient sampling and density estimation. Normalizing flows have become a cornerstone of modern generative modeling \cite{Rezende2015, Kobyzev2020, Papamakarios2021}.

\smallskip\noindent {\sf Advantages of Sampling via Transport Maps.}
Transport maps offer several key advantages for sampling from complex distributions:
\begin{itemize}
    \item \texttt{Efficient Sampling}: Once the map is learned, generating new samples is computationally fast.
    \item \texttt{Handles Complex Distributions}: Transport maps can sample from multi-modal, non-Gaussian, and high-dimensional distributions.
    \item \texttt{Improves MCMC Methods}: By preconditioning distributions, transport maps can accelerate the convergence of Markov Chain Monte Carlo (MCMC) methods \cite{Parno2015, Heng2019}.
    \item \texttt{Provides Density Estimation}: Transport maps allow explicit computation of probability densities, which is useful for tasks like Bayesian inference \cite{Marzouk2016, Spantini2018}.
\end{itemize}

In summary, transport maps provide a powerful and flexible framework for sampling from complex probability distributions by transforming simple source distributions (e.g., Gaussian) into target distributions. 
On one hand, while constructing such maps can be challenging, methods like optimal transport, entropic regularization, and normalizing flows have made significant progress in addressing these difficulties. On the other hand, transport maps are widely used in Bayesian inference, generative modeling, and Monte Carlo methods, offering efficient sampling, density estimation, and improved convergence for MCMC. Despite the challenges, transport maps remain a cornerstone of modern computational statistics and machine learning.

The aim of this work is to provide a simple and computationally tractable approach to constructing transport maps, making them more accessible and easier to implement in practical applications. By focusing on methods that are relatively easy to realize in computations, we aim to bridge the gap between theoretical advancements and practical usability, enabling broader adoption of transport maps in real-world problems.

\subsection{Diffusion Models}

Diffusion models are a class of generative models that have gained significant attention in recent years due to their ability to generate high-quality samples, particularly in tasks such as image synthesis, denoising, and data reconstruction (see \cite{YZS23, CMFW24} for comprehensive surveys). These models are inspired by principles from non-equilibrium thermodynamics and operate by simulating a gradual process of adding noise to data (forward process) and then learning to reverse this process to recover the original data (reverse process). This unique approach allows diffusion models to transform simple noise distributions into complex data distributions, enabling the generation of realistic and diverse samples.

\smallskip\noindent {\sf The Forward Diffusion Process.}
The forward diffusion process systematically corrupts data by adding Gaussian noise over multiple time steps. This process gradually transforms the data into a noise distribution, typically a Gaussian, by following a predefined noise schedule. Mathematically, this can be described by a stochastic differential equation (SDE):
$$
\dif X_t = f_t X_t \dif t + \sigma_t \dif W_t,
$$
where $X_t$ represents the data at time $t$, $f_t$ and $\sigma_t$ are the drift and diffusion coefficients, respectively, and $W_t$ is a standard Brownian motion. By carefully designing $f_t$ and $\sigma_t$, the forward process ensures that the data distribution at the final time step $T$ is approximately Gaussian (see \cite{DWMG15, SD21} for details).

\smallskip\noindent{\sf The Reverse Denoising Process.}
The reverse process is the key to generating new data samples. It learns to denoise the data by reversing the forward diffusion process, step by step. This is achieved by training a neural network to predict the noise added at each time step, allowing the model to iteratively refine noisy data into clean samples. The reverse process can be formulated as another SDE or as a sequence of learned transitions, depending on the specific implementation (see \cite{HJA20, SD21}). Theoretically, the reverse process is governed by the Fokker-Planck equation, which describes the evolution of the probability density of the data over time.

\smallskip\noindent {\sf Advantages of Diffusion Models.}
Diffusion models have several notable advantages:
\begin{itemize}
    \item \texttt{High-Quality Samples}: They are capable of generating highly realistic and diverse samples, making them particularly effective for image synthesis and other generative tasks (see \cite{YZS23, CMFW24}).
    \item \texttt{Theoretical Foundations}: The connection to stochastic processes and differential equations provides a strong theoretical framework for understanding and improving these models (see \cite{DWMG15, SD21}).
    \item \texttt{Flexibility}: Diffusion models can be applied to a wide range of data types, including images, audio, and text (see \cite{HJA20, CLL23}).
\end{itemize}
\noindent {\sf Applications of Diffusion Models.}
One of the most prominent applications of diffusion models is in image generation, where they have been used to produce photorealistic images by sampling from a noise distribution and reversing the diffusion process (see \cite{YZS23, HJA20, CMFW24, LZ22}). They are also highly effective in denoising tasks, where the goal is to recover clean data from noisy inputs (see \cite{SD21, saharia2022}). 

In summary, diffusion models represent a powerful and versatile approach to generative modeling, combining strong theoretical foundations with practical applications in image synthesis, denoising, and beyond. Their ability to generate high-quality samples and their flexibility in handling diverse data types make them a key area of research in modern machine learning. 

Another aim of this paper is to establish a {\it unified and rigorous theoretical foundation} for various diffusion models. In our framework, the traditional procedures of adding noise and denoising are replaced by a direct learning approach. Specifically, we propose using a neural network to learn the transport map from the Gaussian distribution to the data/target distribution. Once the transport map is learned, samples can be generated directly without the need for iterative noise addition and removal. This approach not only simplifies the generative process but also provides a more efficient and theoretically sound  framework for diffusion models.

\subsection{Main Results}
Throughout this paper, we fix the dimension $d \in \mathbb{N}$ and a complete filtered probability space $(\Omega, \mathcal{F}, \mathbb{P}; (\mathcal{F}_t)_{t \geq 0})$. Unless stated otherwise, the expectation with respect to $\mathbb{P}$ is denoted by $\mathbb{E}$, and for $p \in [1, \infty)$, the norm in the $L^p$-space over $(\Omega, \mathcal{F}, \mathbb{P})$ is denoted by $\|\cdot\|_p$.

Let $\xi_0 \sim \mu$ and $\eta \sim \nu$ be two independent random variables. Assume that for some constant $K > 0$, the initial distribution $\mu$ and the target distribution $\nu$ satisfy
\begin{align}\label{C12}
\mu(\dif x) = \rho_0(x) \dif x, \quad 0 < \rho_0(x) \leq K, \quad \mathbb{E}|\xi_0| + \mathbb{E}|\eta| < \infty.
\end{align}
Let $\sigma_t,\beta_t:[0,1]\to[0,1]$ be two $C^1$-functions and satisfy 
\begin{align}\label{Sig0}
\boxed{\beta_t=1-\sigma_t,\ \ \sigma_0=1,\ \ \sigma_1=0,\ \ \sigma'_t<0.}
\end{align}
Below are typical functions of $\sigma_t$ and $\beta_t$ used in practical simulations.
\begin{center}
\begin{tikzpicture}
  \begin{axis}[
    axis lines=middle,
    xlabel={$t$},
    ylabel={$\sigma_t, \beta_t$},
    xmin=0, xmax=1.15,
    ymin=0, ymax=1.15,
    legend style={
      at={(0.5,-0.15)},
      anchor=north,
      draw=none,
      font=\tiny,
      legend columns=6, 
    },
    grid=major,
    samples=100,
    width=8cm,  
    height=5cm, 
  ]
  
    \addplot[
      domain=0:1,
      thick,
      black!70!black,
      dash pattern=on 2pt off 3pt,
    ]
    {1-x};
    \addlegendentry{$\sigma_t = 1-t$}

    \addplot[
      domain=0:1,
      thick,
      gray,
      dash pattern=on 2pt off 3pt,
    ]
    {x};
    \addlegendentry{$\beta_t = t$}
  
    \addplot[
      domain=0:1,
      thick,
      blue,
    ]
    {cos(deg(pi*x/2))^2};
    \addlegendentry{$\sigma_t = \cos^2\left(\frac{\pi t}{2}\right)$}

    \addplot[
      domain=0:1,
      thick,
      orange,
    ]
    {sin(deg(pi*x/2))^2};
    \addlegendentry{$\beta_t = \sin^2\left(\frac{\pi t}{2}\right)$}

    \addplot[
      domain=0:1,
      thick,
      green!70!black,
    ]
    {exp(-x/(1-x))};
    \addlegendentry{$\sigma_t = \e^{-\frac{t}{1-t}}$}

    \addplot[
      domain=0:1,
      thick,
      purple,
    ]
    {1-exp(-x/(1-x))};
    \addlegendentry{$\beta_t = 1-\e^{-\frac{t}{1-t}}$}
    
  \end{axis}
\end{tikzpicture}
\end{center}
For each $t \in [0, 1)$, we define the drift function
\begin{align}\label{BB91}
b_t(x) := \frac{\sigma_t'\mathbb{E}[(x-\eta) \rho_0((x - \beta_t\eta)/\sigma_t)]}
{\sigma_t\mathbb{E} \rho_0((x - \beta_t\eta)/\sigma_t)}.
\end{align}
Under assumption \eqref{C12}, it is easy to see that
$b_t(x)$ is well-defined, and $b_0(x) = \sigma_0'[x-\mathbb{E} \eta]$. Consider the following ordinary differential equation (ODE):
\begin{align}\label{ODE-0}
X_t' = b_t(X_t).
\end{align}
Here, the prime denotes the derivative with respect to the time variable. One of the main results of this paper, which follows from Theorem \ref{Th23}, is stated below.

\begin{theorem}\label{Th01}
Suppose that for some $p \geq 1$, $\mathbb{E}|\xi_0|^p + \mathbb{E}|\eta|^p < \infty$, and for some constants $K, \kappa > 0$, the density $\rho_0\in C^1$ satisfies
\begin{align}\label{SD102}
0 < \rho_0(x) \leq K, \quad |\nabla\rho_0(x)| \leq \kappa(1 + |x|)^{p-1}.
\end{align}
Then, for each initial point $x_0 \in \mathbb{R}^d$, the ODE \eqref{ODE-0} admits a unique solution $(X_t(x_0))_{t \in [0, 1)}$ with $X_0 = x_0$. Moreover, the following properties hold:
\begin{enumerate}[(i)]
\item For each $t\in[0,1)$,
the mapping $x_0\mapsto X_t(x_0)$ is continuously differentiable, 
and the compoistion $X_t(\xi_0)$ follows the distribution $\mu_t(\dif x) = \rho_t(x) \dif x$, where 
$$
\rho_t(x) = \sigma_t^{-d}\mathbb{E} \rho_0((x - \beta_t\eta)/\sigma_t).
$$
\item There exists a measurable map $x_0 \mapsto X_1(x_0)$ such that $X_1(\xi_0) \sim \nu$, and for all $t \in [0, 1]$,
$$
\|X_t(\xi_0) - X_1(\xi_0)\|_p \leq \sigma_t \|\xi_0 - \eta\|_p.
$$
In particular, the map $x_0\mapsto X_1(x_0)$ serves as a {\bf transport map} between $\mu$ and $\nu$. 
\end{enumerate}
\end{theorem}

\begin{remark}
If the target distribution $\nu$ has a density $\rho_1(x)$, 
then by the change of variable $x-y = \sigma_tz$, the drift $b$ can be rewritten as
\begin{align}\label{BB0}
\begin{split}
b_t(x) &= \frac{\sigma_t'\int_{\mathbb{R}^d} (x - y) \rho_0((x - \beta_ty)/\sigma_t) \rho_1(y) \dif y}
{\sigma_t \int_{\mathbb{R}^d} \rho_0((x - \beta_ty)/\sigma_t) \rho_1(y) \dif y}
= \frac{\sigma_t'\int_{\mathbb{R}^d} z \rho_0(x + \beta_tz) \rho_1(x - \sigma_tz) \dif z}
{\int_{\mathbb{R}^d} \rho_0(x + \beta_tz) \rho_1(x - \sigma_tz) \dif z}\\
&= \frac{\beta_t'\int_{\mathbb{R}^d} (x-y) \rho_0(y) \rho_1((x-\sigma_ty)/\beta_t) \dif y}
{\beta_t\int_{\mathbb{R}^d} \rho_0(y) \rho_1((x- \sigma_ty)/\beta_t) \dif y}
=\frac{\beta_t'\mE[(x-\xi_0)\rho_1((x- \sigma_t\xi_0)/\beta_t)]}{\beta_t\mE[\rho_1((x-\sigma_t\xi_0)/\beta_t)]},
\end{split}
\end{align}
and the density $\rho_t(x)$ can be rewritten as
$$
\rho_t(x) = \int_{\mathbb{R}^d} \rho_0\left(\tfrac{x}{\sigma_t} - \beta_ty\right) \rho_1(\sigma_ty) \dif y
= \int_{\mathbb{R}^d} \rho_0(\beta_tz) \rho_1\left(\tfrac{x}{\beta_t} - \sigma_tz\right) \dif z.
$$
In this case, we have
$$
b_0(x)=\sigma_0'[x-\mathbb{E} \eta],\ \ b_1(x)=\beta_1'[x-\mE\xi_0],
$$ 
and $\{\rho_t(x)\}_{t\in[0,1]}$ plays the role of interpolation between $\rho_0(x)$ and $\rho_1(x)$.
Furthermore, if $\rho_1\in C^1$ satisfies \eqref{SD102}, by \eqref{BB0}, $\nabla b_t(x)$ is uniformly bounded on compact sets:
$$
\sup_{t \in [0, 1], |x| \leq R} |\nabla b_t(x)| < \infty, \quad \forall R > 0.
$$
Thus, by standard ODE theory, $X_1(\cdot)$ is also continuously differentiable.
Note that when $\rho_0(x)\propto \e^{-|x|^2/2}$ and $|\eta|\leq K$, 
ODE \eqref{ODE-0} has been considered in \cite{Zh24}.
\end{remark}
\begin{remark}
Let $ d = 1 $ and $\sigma_t=\cos^2(\pi t/2), \beta_t=\sin^2(\pi t/2)$. We choose the initial distribution and the target distribution as follows:
\begin{align}\label{Rho1}
\rho_0(x) \propto (1+|x|^2)^{-1}, \quad \rho_1(x) \propto \left(1 + \frac{\sin(2\pi x) + \sin(4\pi x)}{2}\right) \1_{[0,1]}(x).
\end{align}
Based on ODE \eqref{ODE-0} and Euler's algorithm, we simulate $5000$-trajectories.
The following figure illustrates the transport map $ x \mapsto X_t(x) $ as defined by the ODE in \eqref{ODE-0}:
\begin{itemize}
    \item The left-hand plot displays 5000 simulated trajectories.
    \item The right-hand plot shows the kernel density estimate of the 5000 terminal points.
\end{itemize}
For an initial condition $ \xi_0 \sim \rho_0$, the expected value of $ X_1(\xi_0)\sim\rho_1 $ is given by  
$$
\mathbb{E}[X_1(\xi_0)] = \int_0^1 x \rho_1(x) \, \dif x = \tfrac{1}{2} - \tfrac{3\pi}{8} \approx 0.3804.
$$
The figure confirms that the numerical results align well with this theoretical prediction.
\begin{center}
\begin{minipage}{\textwidth}
\centering
\includegraphics[width=2.5in, height=1.7in]{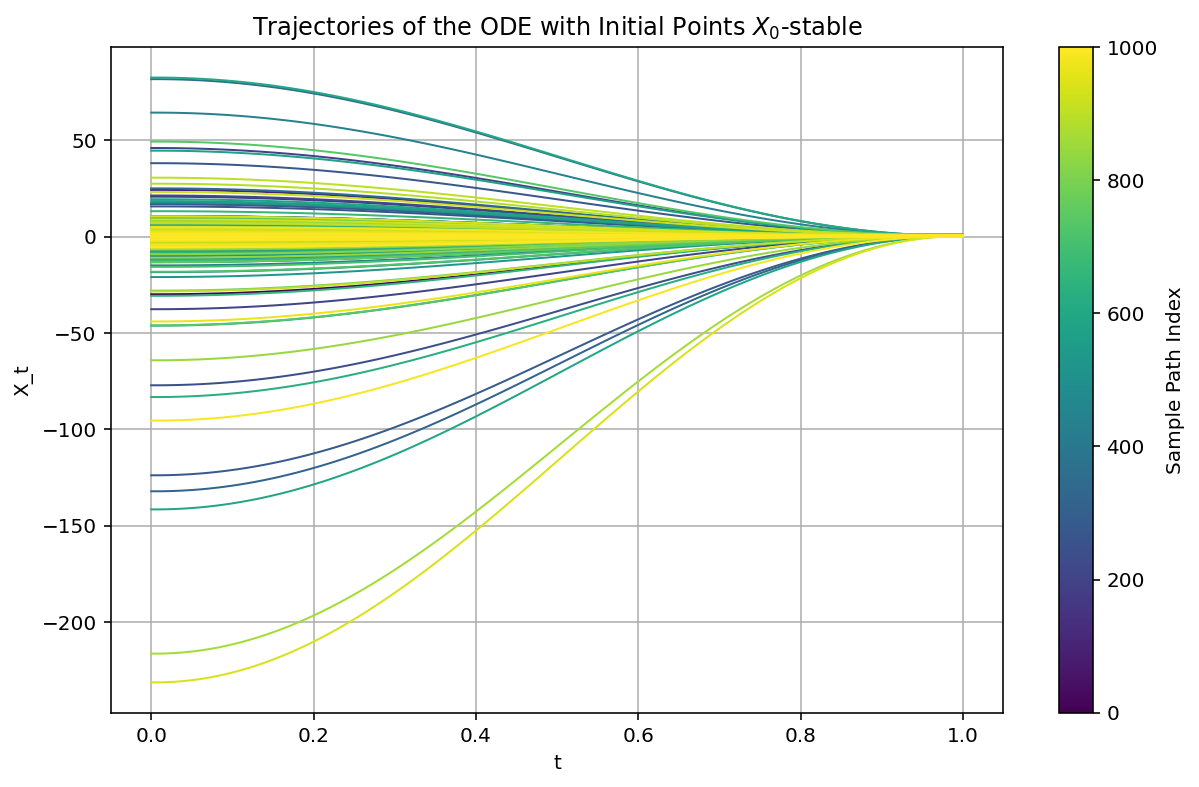}
\includegraphics[width=2.5in, height=1.7in]{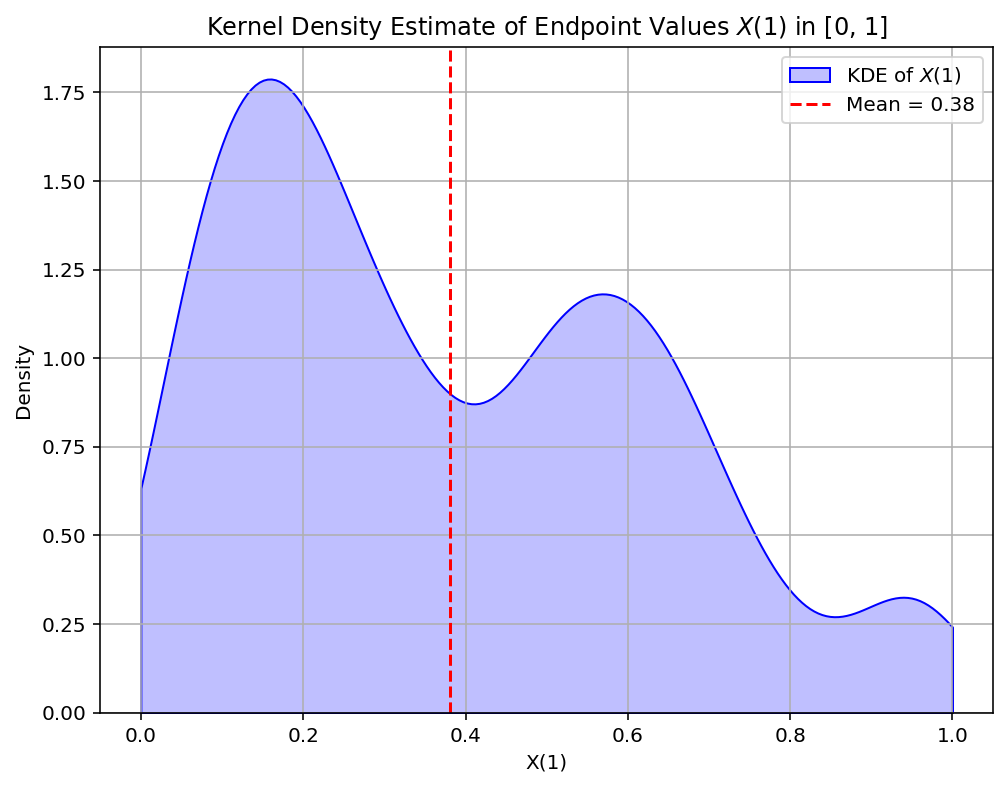}
\captionof{figure}{Deterministic transport map}
\end{minipage}
\end{center}
\end{remark}

The above results require the initial distribution $\mu$ to have a density $\rho_0$ satisfying \eqref{SD102}. To handle more general cases, we introduce a stochastic transport map that can transform any probability distribution $\mu$ into another distribution $\nu$ with compact support. For this purpose, for $\eps>0$, define the drift function
\begin{align}\label{DD108}
b^\eps_t(x) := \frac{\sigma_t'\mathbb{E}[(\eta - x) \e^{-|x - \beta_t\eta -\sigma_t\xi_0|^2/(2\eps \beta_t\sigma_t)}]}
{\sigma_t \mathbb{E} \e^{-|x - \beta_t\eta - \sigma_t\xi_0|^2/(2\eps\beta_t\sigma_t)}}, \quad \forall t \in (0, 1),
\end{align}
where $\xi_0 \sim \mu$ and $\eta \sim \nu$ are two independent random variables. 
Consider the following SDE:
\begin{align}\label{SDE06}
\dif X_t = b^\eps_t(X_t) \dif t + \sqrt{\eps} \dif W_t,
\end{align}
where $W$ is a $d$-dimensional standard Brownian motion independent of $\xi_0$ and $\eta$. The following theorem establishes the existence-uniqueness and properties of solutions to this SDE.

\begin{theorem}\label{Th02}
Suppose that $|\eta| \leq K$ for some constant $K > 0$. Then, for each initial point $x_0 \in \mathbb{R}^d$, the SDE \eqref{SDE06} admits a unique strong solution $(X_t(x_0))_{t \in [0, 1)}$. Furthermore, the following properties hold:
\begin{enumerate}[(i)]
\item The mapping $(t, x_0) \mapsto X_t(x_0)$ is jointly continuous in $t$ and $x_0$ on $[0,1)\times\mR^d$.
\item For each $t \in (0, 1)$ and $x_0 \in \mathbb{R}^d$, $X_t(x_0)$ has a density $p_t(x_0, x)$.
\item For each $t \in (0, 1)$, the random variable of composition $X_t(\xi_0)$ has the density
\begin{align}\label{Den1}
\rho_t(x) = \mathbb{E} \e^{-|x - \beta_t\eta - \sigma_t\xi_0|^2/(2\eps \beta_t\sigma_t)} / (2\pi \eps \beta_t\sigma_t)^{d/2}.
\end{align}
\item There exists a measurable random field $(X_1(x_0))_{x_0 \in \mathbb{R}^d}$ such that 
$$
X_1(\xi_0) \sim \nu,
$$ 
and for all $t \in (0, 1)$,
\begin{align}\label{Con96}
\|X_t(\xi_0) - X_1(\xi_0)\|_2 \leq \sigma_t \|\xi_0 - \eta\|_2 + 2\sqrt{d\eps \sigma_t}.
\end{align}
In particular, the map $x_0 \mapsto X_1(\omega, x_0)$ can be interpreted as a {\bf stochastic transport map}, 
transforming any sample $\xi_0 \sim \mu$ into $X_1(\xi_0) \sim \nu$.
\end{enumerate}
\end{theorem}

\begin{remark}
Note that $\rho_t(x)=\mE p_t(\xi_0,x)$. This together with formula \eqref{Den1} suggests that
$$
p_t(x_0, x) = \mathbb{E} \e^{-|x - \beta_t\eta - \sigma_tx_0|^2/(2\eps \beta_t\sigma_t)} / (2\pi\eps \beta_t\sigma_t)^{d/2}.
$$
However, this is not true because $b_t(x)$ depends on the entire distribution $\mu$ of $\xi_0$.
\end{remark}
\begin{remark}
We consider the case $d=1$ with an initial condition $\xi_0=0$. We also choose 
$\sigma_t=\cos^2(\pi t/2)$ and $\beta_t=\sin^2(\pi t/2)$. Suppose that $\nu(\dif x) = \rho_1(x) \dif x$, where $\rho_1$ is the same as in \eqref{Rho1}.  
Using the first-order expression \eqref{AA3} with $\gamma=0$ for $b^\eps_t(x)$ and the Euler scheme, we simulate 5000 sample paths 
by setting $\eps=0.5$ and time steps $500$, and present the results as follows:  
\begin{itemize}  
    \item The left-hand plot shows 5000 simulated sample paths.  
    \item The right-hand plot depicts the kernel density estimate of the 5000 terminal points.  
\end{itemize}  
From the figure, we observe that the distribution of $X_1(0)$ closely approximates $\nu$.  
\begin{center}
\begin{minipage}{\textwidth}
\centering
\includegraphics[width=2.5in, height=1.7in]{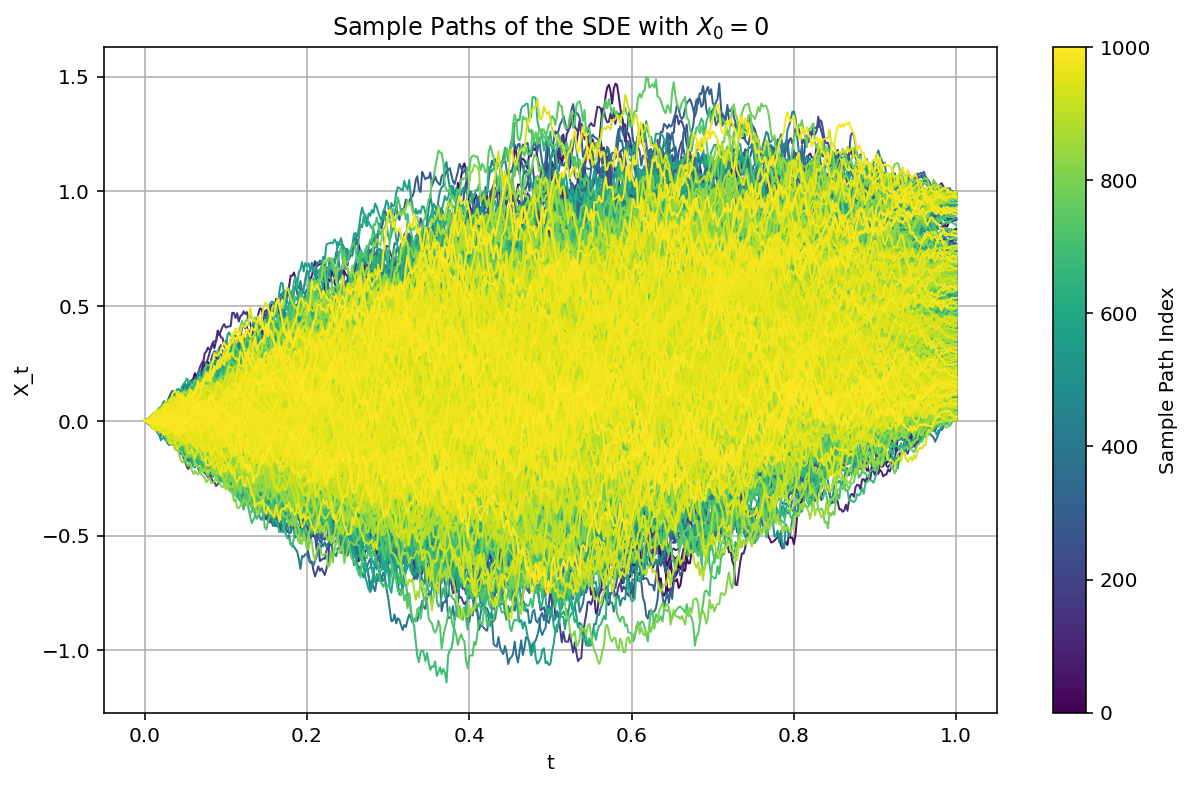}\,\,\,
\includegraphics[width=2.5in, height=1.7in]{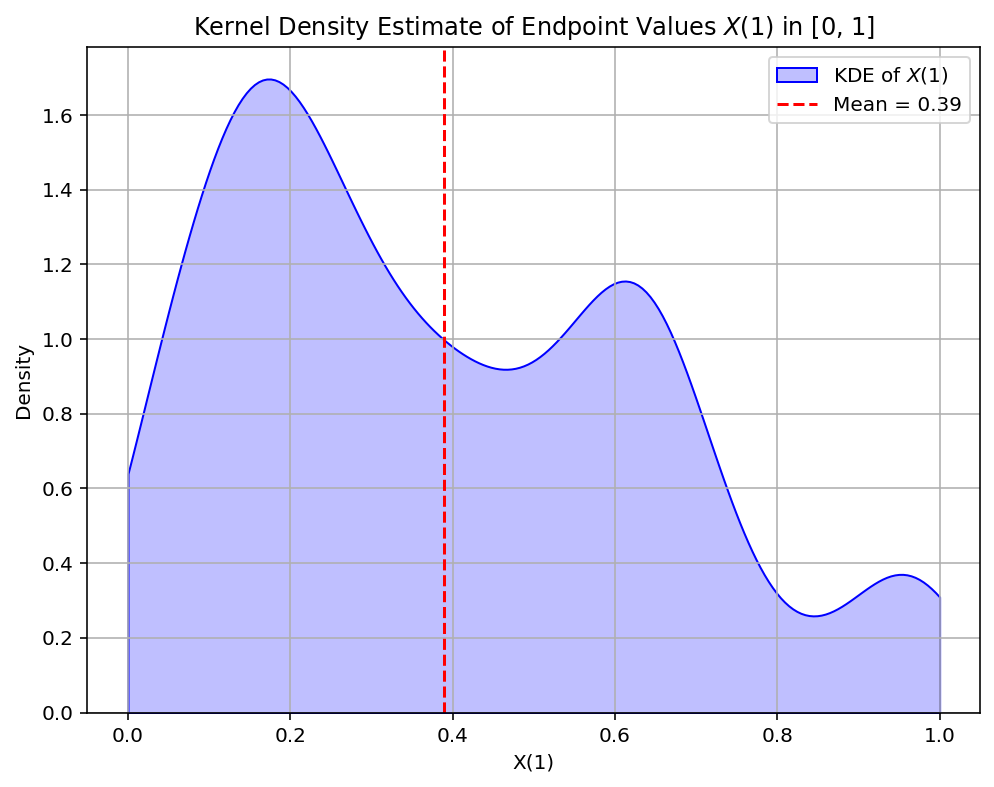}
\captionof{figure}{Stochastic transport map}
\end{minipage}
\end{center}
\end{remark}

\subsection{Related Results and Our Contributions}
We begin by recalling the construction of Moser's mapping (see \cite[Chapter 1, Appendix]{Villani2009}). 
Let $\mu(\dif x) = \rho_0(x) \e^{-V(x)} \dif x$ and $\nu(\dif x) = \rho_1(x)\e^{-V(x)} \dif x$, 
where $V(x) \in C^1(\mathbb{R}^d)$, and $\rho_0, \rho_1$ are two positive, locally Lipschitz functions bounded from below. 
Consider the following elliptic equation:
$$
\Delta u - \nabla V \cdot \nabla u = \rho_0 - \rho_1.
$$
It is well-known that there exists a solution $u \in C^{1,1}_{\text{loc}}$ to the above equation. Define
$$
b(t, x) := \frac{\nabla u(x)}{(1 - t) \rho_0 + t \rho_1(x)},\ \ t\in(0,1).
$$
Let $X_t(x)$ be the ODE flow defined by the vector field $b$. It is easy to see that $\mu_t=(X_t)_\# \mu$ satisfies
$$
\p_t\mu_t=\nabla\cdot (\mu_t X_t)=(\rho_1-\rho_0)\e^{-V}=\nu-\mu\Rightarrow\mu_t = (1 - t) \mu + t \nu.
$$
In particular, $(X_1)_\# \mu = \nu$. Although the construction of this transport map is straightforward, solving the PDE remains challenging.

In a recent work \cite{ABE23}, Albergo, Boffi, and Vanden-Eijnden introduced a general framework for generative models: 
stochastic interpolants. Let $x_0 \sim \rho_0(x) \dif x$, $x_1 \sim \rho_1(x) \dif x$, and $z \sim N(0, \mathbb{I}_d)$ be independent random variables, 
and let $I \in C^2([0, 1]; C^2(\mathbb{R}^d \times \mathbb{R}^d)^d)$ satisfy the conditions
$$
I(0, x_0, x_1) = x_0, \quad I(1, x_0, x_1) = x_1, \quad |\partial_t I(t, x_0, x_1)| \leq C_1 |x_0 - x_1|.
$$
Consider the following interpolant process:
$$
x_t := I(t, x_0, x_1) + \gamma_t z,
$$
where $\gamma: [0, 1] \to \mathbb{R}$ satisfies $\gamma^2 \in C^2([0, 1])$ and $\gamma_0 = \gamma_1 = 0$. 
By calculating the characteristic function, the authors demonstrated that the law $\mu_t$ of $x_t$ has a density $\rho_t$, 
which solves the following transport equation:
$$
\partial_t \rho + \nabla \cdot (b \rho) = 0,
$$
where $b_t(x) := \mathbb{E}\left[\dot{x}_t \mid x_t = x\right]$ is the unique minimizer of the functional
on $C^0([0, 1]; C^1(\mathbb{R}^d)^d)$:
$$
\sL(\hat{b}) := \int^1_0 \mathbb{E}\left[\tfrac{1}{2} |\hat{b}_t(x_t)|^2 - \dot{x}_t \cdot \hat{b}_t(x_t)\right] \dif t
= \int^1_0\!\!\! \int_{\mathbb{R}^d} \left[\tfrac{1}{2} |\hat{b}_t(x)|^2 - b_t(x) \cdot \hat{b}_t(x)\right] \rho_t(x) \dif x \dif t.
$$
Based on this Fokker-Planck equation and the variational representation of $b$, 
the authors provided a unified treatment for various score-based models (see \cite[Section 5]{ABE23} for more discussions).

Compared with \cite{ABE23}, the contributions of this paper are summarized as follows:

\begin{itemize}
    \item We address more general initial and target distributions using It\^o's calculus. 
    Specifically, we derive explicit expressions \eqref{BB91} and \eqref{DD108} for the drift term $b$ in terms of $\mu$ and $\nu$. 
    Notably, although $b_t(x)$ exhibits singular behavior at the terminal time $t=1$, we demonstrate the existence of a (stochastic) transport map $X_1(\cdot)$ up to $t=1$ by leveraging the superposition principle and the special structure of the drift $b_t(x)$.

    \item Utilizing the explicit expression for $b$, we are able to design algorithms with greater precision, moving away from the black-box approach often employed in diffusion models, where the score function or drift $b$ is typically approximated using neural networks.

    \item For sampling problems, we propose three schemes for the drifts: a zero-order scheme, a first-order scheme, and a second-order scheme. 
    We establish the convergence rate of particle approximations by applying the Girsanov theorem and entropy methods (cf. \cite{JW16}). 
    Numerical experiments on high-dimensional anisotropic funnel distributions demonstrate that our schemes achieve excellent performance.
\end{itemize}

\subsection{Organization of this Paper}

This paper is organized as follows:  

In Section 2, we present a general result about the sum of two independent diffusion processes. Specifically, using Itô's formula, we show that the time-marginal distribution of the sum satisfies a Fokker-Planck equation. Furthermore, 
by applying the Ambrosio-Figalli-Trevisan superposition principle, we derive the existence of a solution to the
associated stochastic differential equations (SDEs).

In Section 3, we introduce a general framework for diffusion generative models based on both ordinary differential equations (ODEs) and SDEs. This framework can be viewed as a special case of Corollary \ref{Cor15}. In Theorem \ref{Th23}, we establish the existence of a deterministic transport map via ODEs, while in Theorem \ref{Th311}, we prove the existence of a stochastic transport map. Additionally, we provide several alternative expressions for the drift terms, which are tailored to the target distribution and are particularly useful for sampling applications.

In Section 4, we address the challenge of nonlinear drifts, where the denominators and numerators involve expectations with respect to the target distribution. We demonstrate the optimal convergence rate of the particle approximation (Monte-Carlo method) for the SDEs under the total variation metric. Moreover, when the target distribution has compact support, we establish strong convergence in $L^p$-spaces.

In Section 5, we apply our theoretical results to two concrete sampling problems: high-dimensional anisotropic funnel distributions and Gaussian mixture distributions. Numerical experiments demonstrate the effectiveness of our proposed algorithm.

Finally, in the Appendix, we provide some tools and well-known results that are used in the proofs of our main theorems of convergence of particle approximations.

\subsection{Notations}
We list some notations and definitions frequently used below.
\begin{itemize}
\item For two probability measures $\mu,\nu$ and $p\geq 1$, the $p$-Wasserstein metric is defined by
$$
\mathcal{W}_p(\mu, \nu) := \inf_{\xi \sim \mu, \eta \sim \nu} \|\xi - \eta\|_p,
$$
and if $\mu \ll \nu$, the relative entropy is defined by
\begin{align}\label{Rel1}
\cH(\mu, \nu) := \int_{\mathbb{R}^d} \log(\dif \mu / \dif \nu) \dif \mu.
\end{align}

\item For a signed measure $\mu$ over $\mR^d$, the total variation norm of $\mu$ is defined by
$$
\|\mu\|_{\rm var}:=\sup_{\|f\|_\infty\leq1}\left|\int_{\mR^d}f(x)\mu(\dif x)\right|
$$

\item Let $\fm\in\mR^d$ and $\Sigma$ be a $d\times d$-symmetric positive definite matrix. 
The $d$-dimensional Gaussian density with mean $\fm$ and covariance matrix $\Sigma$  is denoted by
\begin{align}\label{Nor1}
\sN(x|\fm,\Sigma)=\frac{\e^{-\<\Sigma^{-1}_i(x-\fm), x-\fm\>^2/2}}{\sqrt{(2\pi)^d\det(\Sigma)}}.
\end{align}
In particular, for $\sigma>0$, we also simply wrtie $\phi_\sigma(x):=\sN(x|0,\sigma\mI_d)$.
\item For $p\geq 1$ and $d\geq 1$, define a constant 
\begin{align}\label{Cons0}
\mathcal{c}_p := \left[(2\pi)^{-d/2} \int_{\mathbb{R}^d} |x|^p \e^{-|x|^2/2} \, \mathrm{d} x\right]^{1/p}.
\end{align}

\end{itemize}
\section{The Sum of Two Independent Diffusion Processes}

In this section, we present a general result concerning the Fokker-Planck equation satisfied by the sum of two independent diffusion processes, which directly follows from It\^o's formula. Additionally, we establish the existence of a weak solution associated with the Fokker-Planck equation by leveraging the Ambrosio-Figalli-Trevisan superposition principle (see \cite{BRS21}).

Let $\{(X^{(i)}_t)_{t \geq 0}, i = 1, 2\}$ be two stochastic processes that solve the following stochastic differential equations (SDEs):
\begin{align}\label{ES1}
\mathrm{d} X^{(i)}_t = b^{(i)}(t, X^{(i)}_t) \, \mathrm{d} t + \sigma^{(i)}(t, X^{(i)}_t) \, \mathrm{d} W^{(i)}_t, \quad i = 1, 2,
\end{align}
where $W^{(i)}$, $i = 1, 2$, are two independent $d$-dimensional standard Brownian motions, and the coefficients
$$
b^{(i)}(t, \omega, x): \mathbb{R}_+ \times \Omega \times \mathbb{R}^d \to \mathbb{R}^d, \quad 
\sigma^{(i)}(t, \omega, x): \mathbb{R}_+ \times \Omega \times \mathbb{R}^d \to \mathbb{R}^d \otimes \mathbb{R}^d,
$$
are random, adapted processes, and independent of each other for $i=1,2$.

Define the process $X_t$ as the sum of $X^{(1)}_t$ and $X^{(2)}_t$:
$$
X_t := X^{(1)}_t + X^{(2)}_t.
$$
For $i = 1, 2$, let $\mu^{(i)}_t$ denote the law of $X^{(i)}_t$, and let $\mu_t$ denote the law of $X_t$. It follows that
$$
\mu_t(\mathrm{d} x) = (\mu^{(1)}_t * \mu^{(2)}_t)(\mathrm{d} x) = \mathbb{E} \mu^{(1)}_t(\mathrm{d} x - X^{(2)}_t) = \mathbb{E} \mu^{(2)}_t(\mathrm{d} x - X^{(1)}_t).
$$
For simplicity, we define
$$
a^{(i)}_{jk}(t, \omega, x) := \frac{1}{2} \sum_{l=1}^d (\sigma^{(i)}_{jl} \sigma^{(i)}_{kl})(t, \omega, x), \quad i = 1, 2.
$$

We now state the following fundamental result.

\begin{theorem}
Suppose that for any $T > 0$,
\begin{align}\label{DD}
\int^T_0 \left( \mathbb{E} |a^{(i)}(t, X^{(i)}_t)| + \mathbb{E} |b^{(i)}(t, X^{(i)}_t)| \right) \mathrm{d} t < \infty, \quad i = 1, 2.
\end{align}
Then $\mu_t$ solves the following Fokker-Planck equation in the distributional sense:
\begin{align}\label{DD9}
\partial_t \mu_t = \sL^*_t \mu_t,
\end{align}
where $\sL^*_t$ is the adjoint operator of the second-order differential operator $\sL_t$ defined by
$$
\sL_t f(x) := \mathrm{tr}(a(t, x) \cdot \nabla^2 f(x)) + b(t, x) \cdot \nabla f(x),
$$
with the coefficients $a$ and $b$ being non-random and given by
\begin{align}\label{AA}
a(t, x) &:= \mathbb{E}\left[\big(a^{(1)}(t, X^{(1)}_t) + a^{(2)}(t, X^{(2)}_t)\big) \mid X_t = x\right], \\
\label{BB}
b(t, x) &:= \mathbb{E}\left[\big(b^{(1)}(t, X^{(1)}_t) + b^{(2)}(t, X^{(2)}_t)\big) \mid X_t = x\right].
\end{align}
In particular, if we consider the following SDE:
\begin{align}\label{SDE10}
\mathrm{d} X_t = b(t, X_t) \, \mathrm{d} t + \sqrt{2a(t, X_t)} \, \mathrm{d} W_t,
\end{align}
then there exists a solution to the above SDE such that for each $t \geq 0$,
$$
\mu_t = \mathbb{P} \circ X^{-1}_t.
$$
\end{theorem}

\begin{proof}
Let $f \in C^2_b(\mathbb{R}^d)$. By Itô's formula, we have
\begin{align}\label{IT1}
\mathbb{E} f(X_t) = \mathbb{E} f(X_0) + \int^t_0 \mathbb{E}\left[\sL^{(1)}_{s} f(\cdot + X^{(2)}_s)(X^{(1)}_s) + \sL^{(2)}_{s} f(\cdot + X^{(1)}_s)(X^{(2)}_s)\right] \mathrm{d} s,
\end{align}
where for $i = 1, 2$,
$$
\sL^{(i)}_s f(x) = \mathrm{tr}(a^{(i)}(s, \omega, x) \cdot \nabla^2 f(x)) + b^{(i)}(s, \omega, x) \cdot \nabla f(x).
$$
Since $X^{(1)}$ and $X^{(2)}$ are independent, by Fubini's theorem, we have
\begin{align*}
&\mathbb{E}\left[\sL^{(1)}_{s} f(\cdot + X^{(2)}_s)(X^{(1)}_s) + \sL^{(2)}_{s} f(\cdot + X^{(1)}_s)(X^{(2)}_s)\right] \\
&\quad = \mathbb{E}\left[\mathrm{tr}\left((a^{(1)}(s, X^{(1)}_s) + a^{(2)}(s, X^{(2)}_s)) \cdot \nabla^2 f(X_s)\right)\right] \\
&\qquad + \mathbb{E}\left[(b^{(1)}(s, X^{(1)}_s) + b^{(2)}(s, X^{(2)}_s)) \cdot \nabla f(X_s)\right] \\
&\quad = \mathbb{E}\left[\mathrm{tr}(a(s, X_s) \cdot \nabla^2 f(X_s))\right] + \mathbb{E}\left[b(s, X_s) \cdot \nabla f(X_s)\right],
\end{align*}
where $a$ and $b$ are defined by \eqref{AA} and \eqref{BB}, respectively. Substituting this into \eqref{IT1}, we obtain
$$
\int_{\mathbb{R}^d} f(x) \mu_t(\mathrm{d} x) = \int_{\mathbb{R}^d} f(x) \mu(\mathrm{d} x) + \int^t_0 \int_{\mathbb{R}^d} \left[\mathrm{tr}(a(s, x) \cdot \nabla^2 f(x)) + b(s, x) \cdot \nabla f(x)\right] \mu_s(\mathrm{d} x) \mathrm{d} s.
$$
From this, we derive the desired equation \eqref{DD9}.

On the other hand, by \eqref{AA}, \eqref{BB}, and \eqref{DD}, it is easy to see that
$$
\int^T_0 \int_{\mathbb{R}^d} \left(|a(t, x)| + |b(t, x)|\right) \mu_t(\mathrm{d} x) \mathrm{d} t \leq \sum_{i=1,2} \int^T_0 \left(\mathbb{E} |a^{(i)}(t, X^{(i)}_t)| + \mathbb{E} |b^{(i)}(t, X^{(i)}_t)|\right) \mathrm{d} t < \infty.
$$
By the superposition principle for SDEs (cf. \cite{Fi08}, \cite[Theorem 2.5]{Tr13}, or \cite{BRS21}), 
there exists a weak solution to the SDE \eqref{SDE10}.
\end{proof}

In general, it is not known whether the SDE \eqref{SDE10} admits a unique solution even if the SDE \eqref{ES1} admits a unique solution. This is closely related to the regularity of the conditional density $a(t, x)$ and $b(t, x)$. However, we have the following result, which is a direct application of \cite{XXZZ20}.

\begin{theorem}
Let $a$ and $b$ be defined by \eqref{AA} and \eqref{BB}, respectively. Suppose that $a(t, x) = a(t)$ does not depend on $x$, and for some $\kappa_0, \kappa_1, \kappa_2 > 0$,
\begin{align}\label{DD08}
\kappa_0 \mathbb{I} \leq a(t) \leq \kappa_1 \mathbb{I}, \quad |b^{(1)}(t, X^{(1)}_t) + b^{(2)}(t, X^{(2)}_t)| \leq \kappa_2.
\end{align}
Then for each $x \in \mathbb{R}^d$, the SDE
$$
\mathrm{d} X_t = b(t, X_t) \, \mathrm{d} t + \sqrt{2a(t)} \, \mathrm{d} W_t, \quad X_0 = x,
$$
has a unique solution $X_t(x) = X_t$ such that $x \mapsto X_t(x)$ forms a homeomorphism, and
\begin{enumerate}[(i)]
\item for each $T > 0$, it holds that
$$
\sup_{x \in \mathbb{R}^d} \mathbb{E}\left(\sup_{t \in [0, T]} |\nabla X_t(x)|^p\right) < \infty;
$$

\item for each $\xi \sim \mu_0$ independent of $W_\cdot$, it holds that for each $t > 0$,
\begin{align}\label{DD7}
X_t(\xi) \sim \mu_t.
\end{align}
\end{enumerate}
\end{theorem}

\begin{proof}
Under the assumption \eqref{DD08}, we have $|b(t, x)| \leq \kappa_2$. Below we only show (ii) since the other parts follow directly from \cite[Theorem 1.1]{XXZZ20}. Note that $Y_t := X_t(\xi)$ solves the following SDE:
$$
\mathrm{d} Y_t = b(t, Y_t) \, \mathrm{d} t + \sqrt{2a(t)} \, \mathrm{d} W_t, \quad Y_0 = \xi.
$$
By Itô's formula, the law of $Y_t$ solves the Fokker-Planck equation \eqref{DD9}. Moreover, under \eqref{DD08}, it is well known that \eqref{DD9} has a unique solution starting from $\mu_0$ (see \cite{XXZZ20}). Thus, we obtain \eqref{DD7}.
\end{proof}

The following corollary is highly useful and forms the foundation for all subsequent analysis.

\begin{corollary}\label{Cor15}
Suppose that $\mu^{(1)}_t(\mathrm{d} x) = \rho(t, x) \, \mathrm{d} x$ and that $a^{(1)}$ and $b^{(1)}$ are non-random. Then
$$
\mu_t(\mathrm{d} x) / \mathrm{d} x = \mathbb{E} \rho(t, x - X^{(2)}_t) =: \varphi_t(x),
$$
and
$$
\partial_t \varphi_t = \partial^2_{jk} (a_{jk} \varphi_t) - \partial_j (b_j \varphi_t),
$$
where we have used the Einstein convention: if an index appears twice in a product, it is summed automatically. Here,
\begin{align}\label{AA1}
a(t, x) &= \frac{\mathbb{E}\left[(a^{(1)}(t, x - X^{(2)}_t) + a^{(2)}(t, X^{(2)}_t)) \rho(t, x - X^{(2)}_t)\right]}{\mathbb{E} \rho(t, x - X^{(2)}_t)}, \\
\label{BB1}
b(t, x) &= \frac{\mathbb{E}\left[(b^{(1)}(t, x - X^{(2)}_t) + b^{(2)}(t, X^{(2)}_t)) \rho(t, x - X^{(2)}_t)\right]}{\mathbb{E} \rho(t, x - X^{(2)}_t)}.
\end{align}
\end{corollary}

\begin{proof}
For any bounded measurable function $f$, by the change of variable, we have
\begin{align*}
\mathbb{E}\left[a^{(2)}(t, X^{(2)}_t) f(X_t)\right] &= \mathbb{E}\left[a^{(2)}(t, X^{(2)}_t) f(X^{(1)}_t + X^{(2)}_t)\right] \\
&= \int_{\mathbb{R}^d} \mathbb{E}\left[a^{(2)}(t, X^{(2)}_t) f(x + X^{(2)}_t)\right] \rho(t, x) \, \mathrm{d} x \\
&= \int_{\mathbb{R}^d} f(x) \mathbb{E}\left[a^{(2)}(t, X^{(2)}_t) \rho(t, x - X^{(2)}_t)\right] \, \mathrm{d} x,
\end{align*}
and since $a^{(1)}$ is non-random,
\begin{align*}
\mathbb{E}\left[a^{(1)}(t, X^{(1)}_t) f(X_t)\right] &= \mathbb{E}\left[a^{(1)}(t, X^{(1)}_t) f(X^{(1)}_t + X^{(2)}_t)\right] \\
&= \int_{\mathbb{R}^d} \mathbb{E}\left[a^{(1)}(t, x) f(x + X^{(2)}_t)\right] \rho(t, x) \, \mathrm{d} x \\
&= \int_{\mathbb{R}^d} f(x) \mathbb{E}\left[a^{(1)}(t, x - X^{(2)}_t) \rho(t, x - X^{(2)}_t)\right] \, \mathrm{d} x.
\end{align*}
Since $X_t$ has density $\mathbb{E} \rho(t, x - X^{(2)}_t)$, from the above two equalities, it is easy to see that
$$
\mathbb{E}\left[a^{(2)}(t, X^{(2)}_t) \mid X_t = x\right] = \frac{\mathbb{E}\left[a^{(2)}(t, X^{(2)}_t) \rho(t, x - X^{(2)}_t)\right]}{\mathbb{E} \rho(t, x - X^{(2)}_t)},
$$
and
$$
\mathbb{E}\left[a^{(1)}(t, X^{(1)}_t) \mid X_t = x\right] = \frac{\mathbb{E}\left[a^{(1)}(t, x - X^{(2)}_t) \rho(t, x - X^{(2)}_t)\right]}{\mathbb{E} \rho(t, x - X^{(2)}_t)}.
$$
This gives \eqref{AA1} by \eqref{AA}. The expression \eqref{BB1} follows similarly.
\end{proof}

\section{Diffusion Generative Models: ODE and SDE Approaches}

In this section, we present a general framework for various diffusion generative models based on the previous general results about stochastic differential equations (SDEs). Compared with most well-known approaches \cite{SD21, ABE23}, 
our method is more direct and explicit.

Let $\nu \in \mathcal{P}(\mathbb{R}^d)$ be the target distribution that needs to be sampled or generated from the data. Let $\mu \in \mathcal{P}(\mathbb{R}^d)$ be some initial distribution that is easy to sample. We shall construct a dynamical process that evolves $\mu$ to $\nu$ over the time interval $[0, 1]$.

Consider a special case studied in Corollary \ref{Cor15}:
$$
X_t = X^{(1)}_t + X^{(2)}_t = \xi_t + \eta_t,
$$
where $\eta_t: [0, 1] \to \mathbb{R}^d$ is a random $C^1$-function independent of $W$, and $\xi_t$ solves the linear SDE:
$$
\mathrm{d} \xi_t = (\log \sigma_t)' \xi_t \, \mathrm{d} t + \sigma_t \sqrt{a'_t} \, \mathrm{d} W_t,
$$
where $\xi_0$ is independent of $(\eta, W)$, and $\sigma_t, a_t: [0, 1) \to [0, \infty)$ are two $C^1$-functions satisfying
\begin{equation} \label{Con2}
\boxed{
\sigma_0 = 1, \quad \sigma_1 = 0, \quad a_0 = 0, \quad a'_t \geq 0.
}
\end{equation}
It is well-known that $\xi_t$ has the explicit expression:
\begin{equation} \label{GG41}
\xi_t =\e^{\int^t_0 (\log \sigma_r)' \, \mathrm{d} r} \xi_0 + \int^t_0\e^{\int^t_s (\log \sigma_r)' \, \mathrm{d} r} \sigma_s \sqrt{a'_s} \, \mathrm{d} W_s = \sigma_t \xi_0 + \bar{\xi}_t,
\end{equation}
where $\bar{\xi}_t$ is Gaussian and has the covariance matrix (possibly degenerate, e.g. $a_t=0$)
\begin{equation} \label{GG1}
\Theta_t = \mathbb{E}\left(\bar{\xi}_t \bar{\xi}_t^*\right) = \sigma_t^2 (a_t - a_0) \mathbb{I}_d = \sigma_t^2 a_t \mathbb{I}_d,
\end{equation}
where $\mathbb{I}_d$ is the $d \times d$ identity matrix. Below, we always assume that
\begin{equation} \label{AS1}
\boxed{
X_0 \sim \mu, \quad X_1 \sim \nu, \quad \text{Law of } \xi_t = \rho_t(x) \, \mathrm{d} x, \quad \forall t \in (0, 1).
}
\end{equation}

The following result is a direct consequence of Corollary \ref{Cor15}.

\begin{theorem}\label{Th21}
Suppose that \eqref{Con2}, \eqref{AS1}, and
$$
\int^T_0 \mathbb{E}|\eta'_t - (\log \sigma_t)' \eta_t| \, \mathrm{d} t < \infty
$$
hold for each $T < 1$. Then, for each $t \in (0, 1)$, $X_t$ admits a density $\varphi_t(x) = \mathbb{E} \rho_t(x - \eta_t)$, which satisfies the following Fokker-Planck equation:
\begin{equation} \label{PDE1}
\partial_t \varphi_t = \frac{a'_t \sigma_t^2}{2} \Delta \varphi_t - \nabla \cdot (b \varphi_t),
\end{equation}
where
$$
b(t, x) := (\log \sigma_t)' x + \frac{\mathbb{E}[(\eta'_t - (\log \sigma_t)' \eta_t) \rho_t(x - \eta_t)]}{\mathbb{E} \rho_t(x - \eta_t)}.
$$
Moreover, for $p \geq 1$, it holds that for $\mu_t(\mathrm{d} x) = \varphi_t(x) \, \mathrm{d} x$,
\begin{equation} \label{WW1}
\mathcal{W}_p(\mu_t, \nu) \leq \sigma_t \left(\|\xi_0\|_p + \mathcal{c}_p \sqrt{a_t}\right) + \|\eta_t - \eta_1\|_p,
\end{equation}
where $\mathcal{W}_p(\mu, \nu)$ is the Wasserstein metric of $p$-order, and $\mathcal{c}_p$ is defined in \eqref{Cons0}.
\end{theorem}

\begin{proof}
Since $X_t$ is the sum of two independent random variables $\xi_t$ and $\eta_t$, 
under \eqref{AS1}, the density of $X_t$ is given by the convolution
$$
\varphi_t(x) = \mathbb{E} \rho_t(x - \eta_t).
$$
Moreover, by Corollary \ref{Cor15}, $\varphi_t$ solves the following Fokker-Planck equation:
$$
\partial_t \varphi_t = \partial^2_{jk} (a_{jk} \varphi_t) - \partial_j (b_j \varphi_t),
$$
with
$$
a(t, x) = \frac{\mathbb{E}[a'_t \sigma_t^2 \rho_t(x - \eta_t)]}{2 \mathbb{E} \rho_t(x - \eta_t)} \mathbb{I}_d = \frac{a'_t \sigma_t^2}{2} \mathbb{I}_d,
$$
and
$$
b(t, x) = \frac{\mathbb{E}[((\log \sigma_t)' (x - \eta_t) + \eta'_t) \rho_t(x - \eta_t)]}{\mathbb{E} \rho_t(x - \eta_t)}.
$$
This gives the desired PDE \eqref{PDE1}. As for \eqref{WW1}, by definition, we have
$$
\mathcal{W}_p(\mu_t, \nu) \leq \|\sigma_t \xi_0 + \bar{\xi}_t + \eta_t - \eta_0\|_p \leq \sigma_t \|\xi_0\|_p + \|\bar{\xi}_t\|_p + \|\eta_t - \eta_0\|_p,
$$
which, together with $\|\bar{\xi}_t\|_p = \sigma_t \sqrt{a_t} \mathcal{c}_p$, 
where $\mathcal{c}_p$ is defined in \eqref{Cons0}, yields \eqref{WW1}.
\end{proof}


\begin{remark}
The estimate \eqref{WW1} provides the convergence rate of $\mu_t \to \nu$ as $t \to 1$. Notice that the right-hand side of \eqref{WW1} depends on the dimension. For each component $i = 1, \dots, d$, letting $\mu^i_t$ be the law of $X^i_t$, we also have
$$
\mathcal{W}_2(\mu^i_t, \mu^i_1) \leq \sigma_t \|\xi^i_0\|_2 + \sigma_t \sqrt{a_t} + \|\eta^i_t - \eta^i_0\|_2,
$$
which provides the convergence rate for each component.
\end{remark}

\begin{remark}
Suppose that $\xi_0$ has a density or
\begin{equation} \label{LL}
\sigma^2_t a_t > 0 \quad \text{for each } t \in (0, 1).
\end{equation}
Then, it is easy to see that condition \eqref{AS1} holds. We shall discuss these two cases below.
\end{remark}

\subsection{Forward ODE Approach: Starting from Absolutely Continuous Distributions}
In this subsection, we use forward ODE starting from any absolutely continuous distribution to generate the sample.
For this aim, we assume that 
$$
\boxed{\mu(\dif x)=\rho_0(x)\dif x,\  a_t=0, \  \eta_t=\beta_t\eta,\ \eta\sim\nu,}
$$
where $\rho_0$ is a $C^1$-density function (for examples, Gaussian or stable distribution).
Suppose that $\sigma_t, \beta_t\in C^1([0,1])$ satisfy \eqref{Sig0}. In this case, the density of $\xi_t=\sigma_t\xi_0$ is given by
$$
\rho_t(x)=\rho_0(x/\sigma_t)/\sigma_t^d.
$$
In particular, by Theorem \ref{Th21}, the density $\varphi_t(x)=\mE\rho_t(x-\beta_t\eta)$ 
solves the  PDE
$$
\p_t\varphi_t=-\div(b_t\varphi_t),
$$
with
\begin{align}\label{SD02}
b_t(x):=(\log\sigma_t)'[x-\cD_t(x)],
\end{align}
where
\begin{align}\label{DD1}
\cD_t(x)=\frac{\mE[\eta\rho_t(x-\beta_t\eta)]}{\mE\rho_t(x-\beta_t\eta)}
=\frac{\mE[\eta\rho_0((x-\beta_t\eta)/\sigma_t)]}{\mE\rho_0((x-\beta_t\eta)/\sigma_t)}.
\end{align}
\br
$\cD_t(x)$ is usually not bounded and Lipschitz continuous except $\eta$ is bounded.
Moreover, from the definition of $b_t(x)$, one sees that $b_0(x)=\sigma_0'[x-\mE\eta]$, and $b_1(x)$ in general does not make sense since $\sigma_1=0$.
\er


Now we consider the following ODE:
\begin{align}\label{ODE10}
X'_t=b_t(X_t).
\end{align}
We have the following well-posedness result under mild assumptions on $\mu$ and $\nu$.
\bt\label{Th23}
Suppose that \eqref{Sig0} holds and  $\mE|\xi_0|^p+\mE|\eta|^p<\infty$ for some $p\geq 1$.
\begin{enumerate}[(i)]
\item 
For any $t_0\in(0,1)$, there is a solution $(X_t)_{t\in[0,t_0]}$ to ODE \eqref{ODE10} with $X_0=\xi_0$ and
\begin{align}\label{SQ1}
(\mP\circ X^{-1}_t)(\dif x)=\mu_t(\dif x)=\varphi_t(x)\dif x=\mE\rho_t(x-\beta_t\eta)\dif x.
\end{align}
Moreover, for each $t\in[0,t_0]$, 
\begin{align}\label{Mo11}
\left\|X_t\right\|_p\leq  \sigma_t\|\xi_0\|_p+\beta_t\|\eta\|_p,
\end{align}
and
\begin{align}\label{WW3}
\cW_p(\mu_{t},\nu)\leq \sigma_{t}\|\xi_0-\eta\|_p.
\end{align}
\item Suppose that  for some $K,\kappa>0$,
\begin{align}\label{SD12}
0<\rho_0(x)\leq K,\ \ |\nabla\rho_0(x)|\leq\kappa(1+|x|)^{p-1}.
\end{align}
Then for each $x_0\in\mR^d$, ODE \eqref{ODE10} has a unique solution $(X_t(x_0))_{t\in[0,1)}$ with $X_0(x_0)=x_0$ and
\begin{align}\label{GZ2}
X_t(\cdot)\in C^1,\ X_t(\xi_0)\sim\mu_t(\dif x)=\varphi_t(x)\dif x,\ \ t\in[0,1).
\end{align}
\item Under \eqref{SD12},
there is a measurable mapping $x_0\mapsto X_1(x_0)$ such that for all $t\in[0,1]$,
$$
\|X_t(\xi_0)-X_1(\xi_0)\|_p\leq \sigma_t\|\xi_0-\eta\|_p.
$$
\end{enumerate}
\et
\begin{proof}
(i) Fix $t_0\in(0,1)$. Since $\varphi_t(x)=\mE\rho_t(x-\beta_t\eta)$, by Fubini's theorem, we have
\begin{align*}
\int^{t_0}_0\!\!\int_{\mR^d}|b_t(x)|\varphi_t(x)\dif x
&\leq \int^{t_0}_0|(\log\sigma_t)'|\left[\int_{\mR^d}\Big(|x|\varphi_t(x)+
\mE[|\eta|\rho_t(x-\beta_t\eta)]\Big)\dif x\right]\dif t\\
&=\int^{t_0}_0|(\log\sigma_t)'|\left(\mE|\sigma_t\xi_0+\beta_t\eta|+\mE|\eta|\right)\dif t\\
&\leq \left(\mE|\xi_0|+2\mE|\eta|\right)\int^{t_0}_0|(\log\sigma_t)'|\dif t
=\left(\mE|\xi_0|+2\mE|\eta|\right)\log\sigma^{-1}_{t_0}.
\end{align*}
By the superposition principle (see \cite[Theorem 2.5]{Tr13}),
there is at least one solution to ODE \eqref{ODE10} such that for each $t\in[0,t_0]$, \eqref{SQ1} holds.
For \eqref{Mo11}, by the chain rule and \eqref{ODE10} we have
\begin{align*}
\dif \left(\frac{X_t}{\sigma_t}\right)
=\frac{\dif X_t}{\sigma_t}-\frac{\sigma'_tX_t}{\sigma_t^2} \dif t
&=\frac{b_t(X_t)}{\sigma_t}\dif t-\frac{\sigma'_t X_t}{\sigma_t^2} \dif t
=-\frac{\sigma'_t\cD_t(X_t)}{\sigma^2_t}\dif t,
\end{align*}
which implies by $\beta'_t=-\sigma'_t>0$ that
\begin{align*}
\left\|\sup_{s\in[0,t]}\frac{|X_s|}{\sigma_s}\right\|_p
\leq\|\xi_0\|_p+\int^t_0
\frac{\beta_s'\|\cD_s(X_s)\|_p}{\sigma_s^{2}}\dif s.
\end{align*}
Recalling \eqref{DD1}, by \eqref{SQ1}, Jensen's inequality and Fubini's theorem, we have
\begin{align*}
\mE|\cD_s(X_s)|^p&=\int_{\mR^d}|\cD_s(x)|^p\varphi_s(x)\dif s
\leq \int_{\mR^d}\frac{\mE[|\eta|^p\rho_t(x-\beta_s\eta)]}
{\mE\rho_s(x-\beta_s\eta)}\varphi_s(x)\dif x\\
&=\mE\left[|\eta|^p\int_{\mR^d}\rho_s(x-\beta_s\eta)\dif x\right]=\mE|\eta|^p.
\end{align*}
Hence,
$$
\left\|X_t\right\|_p
\leq \sigma_t\left\|\sup_{s\in[0,t]}\frac{|X_s|}{\sigma_s}\right\|_p
\leq\sigma_t\|\xi_0\|_p+\sigma_t\|\eta\|_p\int^t_0\frac{\beta'_s}{\sigma_s^2}\dif s.
$$
The estimate \eqref{Mo11} now follows by noting that
\begin{align}\label{DS13}
\int^t_0\frac{\beta_s'}{\sigma_s^2}\dif s=-\int^t_0\frac{\sigma_s'}{\sigma_s^2}\dif s =\frac{1}{\sigma_t}-1=\frac{\beta_t}{\sigma_t}.
\end{align}
As for \eqref{WW3}, by the definition of Wasserstein metric, we have
$$
\cW_p(\mu_t,\nu)\leq \|\sigma_t\xi_0+\beta_t\eta-\eta\|_p=\sigma_t\|\xi_0-\eta\|_p.
$$
(ii) To show the uniquness, by the chain rule, we have
\begin{align*}
\nabla_x\cD_t(x)&=\frac{\mE[\eta\otimes\nabla\rho_0((x-\beta_t\eta)/\sigma_t)]}{\sigma_t\mE\rho_0((x-\beta_t\eta)/\sigma_t)}
-\frac{\mE[\eta\rho_0((x-\beta_t\eta)/\sigma_t)]\otimes\mE[\nabla\rho_0((x-\beta_t\eta)/\sigma_t)]}{\sigma_t(\mE\rho_0((x-\beta_t\eta)/\sigma_t))^2}.
\end{align*}
By \eqref{SD12}, it is easy to see that
$$
|\nabla_x\cD_t(x)|\leq
\frac{\kappa\mE[|\eta|(1+|x-\beta_t\eta|/\sigma_t)^{p-1}]}{\sigma_t\mE\rho_0((x-\beta_t\eta)/\sigma_t)}
+\frac{K\kappa\mE|\eta|\,\mE[(1+|x-\beta_t\eta|/\sigma_t)^{p-1}]}{\sigma_t(\mE\rho_0((x-\beta_t\eta)/\sigma_t))^2},
$$
and since $\rho_0(x)>0$ is continuous, 
by the dominated convergence theorem, the mapping
$$
(t,x)\mapsto\mE\rho_0((x-\beta_t\eta)/\sigma_t)
$$ 
is continuous on $[0,1)\times\mR^d$ and strictly positive. In particular, 
for each $R>0$ and $t_0\in(0,1)$,
$$
\inf_{t\in[0,t_0], |x|\leq R}\mE\rho_0((x-\beta_t\eta)/\sigma_t)>0.
$$
Thus, since $\mE|\eta|^p<\infty$, we have for each $t_0\in[0,1)$ and $R>0$,
\begin{align}\label{AS71}
\sup_{t\in[0,t_0], |x|\leq R}|\nabla_x\cD_t(x)|<\infty.
\end{align}
By \cite[Proposition 8.1.8]{AGS05}, there is a Lebesgue null set $E\subset\mR^d$ such that 
for each $x_0\notin E$,  ODE \eqref{ODE10}
admits a unique solution $(X_t(x_0))_{t\in[0,1)}$ so that $X_t(\xi_0)\sim\mu_t$. It remains to show that for all $x_0\in\mR^d$, ODE \eqref{ODE10} admits a unique solution $(X_t(x_0))_{t\in[0,1)}$ 
and $X_t(\cdot)\in C^1$, which follows by \eqref{AS71} and \cite[page 95, Theorem 3.1]{H64}.

(iii) Fix $0\leq t_0<t_1<1$. By \eqref{GZ2}, the definition of $b_s(x)$ and Jensen's inequality, we have
\begin{align*}
\|X_{t_1}(\xi_0)-X_{t_0}(\xi_0)\|_p&\leq\int^{t_1}_{t_0}\|b_s(X_s(\xi_0))\|_p\dif s
=\int^{t_1}_{t_0}\left(\int_{\mR^d}|b_s(x)|^p\varphi_s(x)\dif x\right)^{1/p}\dif s\\
&\leq\int^{t_1}_{t_0}|(\log\sigma_s)'|\left(\int_{\mR^d}\frac{\mE\big[|x-\eta|^p\rho_s(x-\beta_s\eta)\big]}
{\mE\rho_s(x-\beta_s\eta)}\varphi_s(x)\dif x\right)^{1/p}\dif s\\
&=\int^{t_1}_{t_0}|(\log\sigma_s)'|\left(\int_{\mR^d}\!\int_{\mR^d}|x-y|^p\rho_s(x-\beta_sy)\nu(\dif y)\dif x\right)^{1/p}\dif s.
\end{align*}
By the change of variable $x-\beta_sy=\sigma_s z$ and $\sigma_s+\beta_s=1$, and recalling $\rho_s(x)=\rho_0(x/\sigma_s)/\sigma_s^d$, we further have
\begin{align}\label{SQ12}
\|X_{t_1}(\xi_0)-X_{t_0}(\xi_0)\|_p
&\leq\int^{t_1}_{t_0}|\sigma_s'|\left(\int_{\mR^d}\int_{\mR^d}|z-y|^p\rho_0(z)\nu(\dif y)\dif z\right)^{1/p}\dif s\no\\
&=\|\xi_0-\eta\|_p\int^{t_1}_{t_0}|\sigma_s'|\dif s=\|\xi_0-\eta\|_p(\sigma_{t_0}-\sigma_{t_1}),
\end{align}
which implies that
$$
\lim_{t_0,t_1\uparrow 1}\int_{\mR^d}|X_{t_1}(x)-X_{t_0}(x)|^p\rho_0(x)\dif x=0.
$$
From these we conclude (iii).
\end{proof}

\br
Suppose that $\rho_0(x)=\e^{-|x|^2/2}/(2\pi)^{d/2}$ is the Gaussian density. Then
$$
b_t(x)=\frac{\sigma'_t\mE[(x-\eta)\e^{-|x-\beta_t\eta|^2/(2\sigma^2_t)}]}{\sigma_t\mE\e^{-|x-\beta_t\eta|^2/(2\sigma^2_t)}}.
$$
This case has been studied in \cite{Zh24}, where $\eta$ is assumed to be bounded. 
\er

\br
Suppose that $\nu$ has a density $\rho_1(x)$. 
Note that
$$
b_t(x)=\frac{\sigma'_t\mE[(x-\eta) \rho_0((x-\beta_t\eta)/\sigma_t)]}{\sigma_t\mE \rho_0((x-\beta_t\eta)/\sigma_t)}
=\frac{\sigma'_t\int_{\mR^d}(x-y) \rho_0((x-\beta_ty)/\sigma_t)\rho_1(y)\dif y}{\sigma_t\int_{\mR^d} \rho_0((x-\beta_ty)/\sigma_t) \rho_1(y)\dif y}.
$$
By the change of variable $x-y=\sigma_t z$, we obtain
$$
b_t(x)=\frac{\sigma_t'\int_{\mR^d}z\rho_0(x+\beta_tz)\rho_1(x-\sigma_t z)\dif z}{\int_{\mR^d}\rho_0(x+\beta_tz)\rho_1(x-\sigma_t z)\dif z}.
$$
To calculate the above integral, we can use Monte-Carlo's method.
Let $\xi\sim N(0,\mI_d)$. We can rewrite the above expression as
$$
b_t(x)=\frac{\sigma_t'\mE[\xi\rho_0(x+\beta_t\xi)\e^{|\xi|^2/2}\rho_1(x-\sigma_t \xi)]}{\mE[\rho_0(x+\beta_t\xi)\e^{|\xi|^2/2}\rho_1(x-\sigma_t \xi)]}.
$$
For given $N\in\mN$,
let $\xi_1,\cdots,\xi_N$ be i.i.d. standard normal random variables. We introduce the following approximation of $b_t(x)$:
$$
b^N_t(x)=\frac{\sigma'_t\sum_{i=1}^N[\xi_i\rho_0(x+\beta_t\xi_i)\e^{|\xi_i|^2/2}\rho_1(x-\sigma_t\xi_i)]}
{\sum_{i=1}^N[\rho_0(x+\beta_t\xi_i)\e^{|\xi_i|^2/2}\rho_1(x-\sigma_t\xi_i)]},
$$
and consider the approximation ODE:
$$
\dif X_t^N/\dif t= b_t(X^N_t),\ \ X^N_0\sim N(0,\mI_d).
$$
The convergence of the corresponding processes will be studied in Section 4.
\er

\subsection{Forward SDE Approach: Starting from Any Distribution}\label{Sec31}

In this subsection, we discuss a direct forward stochastic differential equation (SDE) approach to generate samples. For this purpose, we choose the functions used in $X_t$ as follows:
$$
\boxed{\eta_t = \beta_t \eta, \quad \eta \sim \nu, \quad a_t = \epsilon \beta_t/\sigma_t,}
$$
where $\epsilon > 0$ represents the noise intensity, and $\sigma_t$, $\beta_t$ satisfy Assumption \eqref{Sig0}. 
In this case, by \eqref{GG1}, the density of $\xi_t$ is given by
\begin{align}\label{SQ13}
\rho_t(x) = \mathbb{E} \left[ \phi_{\epsilon \beta_t \sigma_t}(x - \sigma_t \xi_0) \right] = \mathbb{E} \left[ \frac{\e^{-|x - \sigma_t \xi_0|^2 / (2 \epsilon \beta_t \sigma_t)}}{(2 \pi \epsilon \beta_t \sigma_t)^{d/2}} \right],
\end{align}
and
$$
\sqrt{a'_t} \sigma_t = \sqrt{\epsilon (\beta'_t \sigma_t - \beta_t \sigma'_t)} = \sqrt{\epsilon \beta'_t}.
$$
In particular, by \eqref{PDE1}, $\varphi_t(x) = \mathbb{E} \rho_t(x - \beta_t \eta)$ satisfies the following second-order PDE:
$$
\partial_t \varphi_t = \frac{\epsilon \beta'_t}{2} \Delta \varphi_t - \nabla \cdot (b \varphi_t),
$$
where
\begin{align}\label{SD2}
b_t(x) := (\log \sigma_t)' \left[ x - \mathcal{D}_t(x) \right],
\end{align}
and, by \eqref{SQ13} and the independence of $\xi_0$ and $\eta$,
\begin{align}\label{DD8}
\mathcal{D}_t(x) := \frac{\mathbb{E} \left[ \eta \rho_t(x - \beta_t \eta) \right]}{\mathbb{E} \rho_t(x - \beta_t \eta)} 
= \frac{\mathbb{E} \left[ \eta \phi_{\epsilon \beta_t \sigma_t}(x - \beta_t \eta - \sigma_t \xi_0) \right]}{\mathbb{E} \phi_{\epsilon \beta_t \sigma_t}(x - \beta_t \eta - \sigma_t \xi_0)}
= \frac{\mathbb{E} \left[ \eta\e^{-|x - \beta_t \eta - \sigma_t \xi_0|^2 / (2 \epsilon \beta_t \sigma_t)} \right]}{\mathbb{E}\e^{-|x - \beta_t \eta - \sigma_t \xi_0|^2 / (2 \epsilon \beta_t \sigma_t)}}.
\end{align}
\br
Note that $\mathcal{D}_0(x)$ and $\mathcal{D}_1(x)$ are not well-defined due to $\beta_0 = \sigma_1 = 0$.
\er
Now, we consider the following SDE associated with the above Fokker-Planck equation:
\begin{align}\label{SDE0}
\dif X_t = b_t(X_t) \dif t + \sqrt{\epsilon \beta'_t} \dif W_t, \quad X_0 = \xi_0.
\end{align}
To solve this SDE, we first establish the following lemma, which will be used to prove uniqueness.

\bl\label{Le34}
\begin{enumerate}[(i)]
\item Suppose $|\eta| \leq K$ and $\mathbb{E} |\xi_0| < \infty$. Then, for all $t \in (0, 1)$ and $x \in \mathbb{R}^d$,
\begin{align}\label{KJ681}
|\mathcal{D}_t(x)| \leq K, \quad |\nabla \mathcal{D}_t(x)| \leq 2 K^2/(\epsilon \sigma_t).
\end{align}

\item Suppose $\mathbb{E}\e^{\lambda |\eta|} < \infty$ for any $\lambda > 0$ and $\xi_0 = x_0$ is non-random. Then,
\begin{align}\label{KJ381}
\sup_{t \in [0, t_0], |x| \leq R} |\nabla \mathcal{D}_t(x)| < \infty, \quad \forall R > 0, t_0 \in (0, 1).
\end{align}

\item Suppose $\mu(\dif x) = \rho_0(x) \dif x$, and for some $p \geq 1$, $\mathbb{E} |\eta|^p < \infty$, and for some $K, \kappa > 0$,
$$
0 < \rho_0(x) \leq K, \quad |\nabla \rho_0(x)| \leq \kappa (1 + |x|)^{p-1}.
$$
Then, the estimate \eqref{KJ381} still holds.
\end{enumerate}
\el

\begin{proof}
(i) For simplicity, let
$$
\mathcal{E}_t(x, \eta) :=\e^{-|x - \beta_t \eta - \sigma_t \xi_0|^2 / (2 \epsilon \beta_t \sigma_t)}.
$$
By the chain rule, we have
\begin{align*}
\nabla \mathcal{D}_t(x) = \frac{\mathbb{E} \left[ (\eta \otimes \eta) \mathcal{E}_t(x, \eta) \right]}{\epsilon \sigma_t \mathbb{E} \mathcal{E}_t(x, \eta)} 
- \frac{\mathbb{E} \left[ \eta \mathcal{E}_t(x, \eta) \right] \otimes \mathbb{E} \left[ \eta \mathcal{E}_t(x, \eta) \right]}{\epsilon \sigma_t (\mathbb{E} \mathcal{E}_t(x, \eta))^2}.
\end{align*}
If $\eta$ is bounded by $K$, then \eqref{KJ681} follows directly.

(ii) Observing that $\beta_t = 1 - \sigma_t$, we have
\begin{align}\label{BC8}
\begin{split}
|x - \beta_t \eta - \sigma_t \xi_0|^2 
&= |x|^2 + \beta_t^2 |\eta|^2 + \sigma_t^2 |\xi_0|^2 
- 2 \beta_t \langle x, \eta \rangle - 2 \sigma_t \langle x, \xi_0 \rangle + 2 \beta_t \sigma_t \langle \eta, \xi_0 \rangle \\
&= \beta_t |x - \eta|^2 + \sigma_t |x - \xi_0|^2 - \beta_t \sigma_t |\eta - \xi_0|^2.
\end{split}
\end{align}
For $\xi_0 = x_0$ non-random, by eliminating the common factor $e^{-|x - x_0|^2 / (2 \epsilon \beta_t)}$ from the numerator and denominator in \eqref{DD8}, 
we obtain, for $t \in (0, 1)$,
\begin{align}\label{DS3}
\mathcal{D}_t(x) = \frac{\mathbb{E} \left[ \eta\e^{(\sigma_t |\eta - x_0|^2 - |x - \eta|^2) / (2 \epsilon \sigma_t)} \right]}{\mathbb{E}\e^{(\sigma_t |\eta - x_0|^2 - |x - \eta|^2) / (2 \epsilon \sigma_t)}}
= \frac{\mathbb{E} \left[ \eta\e^{(2 \langle \eta, x - \sigma_t x_0 \rangle - \beta_t |\eta|^2) / (2 \epsilon \sigma_t)} \right]}{\mathbb{E}\e^{(2 \langle \eta, x - \sigma_t x_0 \rangle - \beta_t |\eta|^2) / (2 \epsilon \sigma_t)}}.
\end{align}
For simplicity, let
$$
\mathcal{C}_t(x, \eta) :=\e^{(2 \langle \eta, x - \sigma_t x_0 \rangle - \beta_t |\eta|^2) / (2 \epsilon \sigma_t)}.
$$
By the chain rule, it is easy to see that
\begin{align*}
\nabla_x \mathcal{D}_t(x) 
&= \frac{\mathbb{E} \left[ (\eta \otimes \eta) \mathcal{C}_t(x, \eta) \right]}{\epsilon \sigma_t \mathbb{E} \mathcal{C}_t(x, \eta)} 
- \frac{\mathbb{E} \left[ \eta \mathcal{C}_t(x, \eta) \right] \otimes \mathbb{E} \left[ \eta \mathcal{C}_t(x, \eta) \right]}{\epsilon \sigma_t (\mathbb{E} \mathcal{C}_t(x, \eta))^2}.
\end{align*}
Since $\mathbb{E}\e^{\lambda |\eta|} < \infty$ for any $\lambda > 0$, by the dominated convergence theorem, 
the function $[0, 1) \times \mathbb{R}^d \ni (t, x) \mapsto \mathbb{E} \mathcal{C}_t(x, \eta)$ is continuous and strictly positive. Hence,
$$
\inf_{t \in [0, t_0], |x| \leq R} \mathbb{E} \mathcal{C}_t(x, \eta) > 0.
$$
The desired estimate \eqref{KJ381} follows.

(iii) If $\mu(\dif x) = \rho_0(x) \dif x$, then by \eqref{AS1} and \eqref{GG41},
$$
\mathcal{D}_t(x) = \frac{\mathbb{E} \left[ \eta \rho_t(x - \beta_t \eta) \right]}{\mathbb{E} \rho_t(x - \beta_t \eta)} 
= \frac{\mathbb{E} \left[ \eta \rho_0((x - \beta_t \eta - \bar{\xi}_t) / \sigma_t) \right]}{\mathbb{E} \rho_0((x - \beta_t \eta - \bar{\xi}_t) / \sigma_t)}.
$$
As in the proof of \eqref{AS71}, \eqref{KJ381} holds.
\end{proof}

\br
The assumption of exponential moments for $\eta$ in (ii) ensures the continuity of $\mathcal{D}_t(x)$ at $t = 0$. 
Note that we divide the factor $\e^{-|x - x_0|^2 / (2 \epsilon \beta_t)}$ in \eqref{DS3} from the numerator and denominator, 
which vanishes at $t = 0$, so that $\mathbb{E} \mathcal{C}_t(x, \eta)$ is well-defined at $t = 0$. 
In other words, the continuity of $\mathcal{D}_t(x)$ at $t = 0$ cannot be directly deduced from \eqref{DD8}.
\er

We now present and prove the following result.

\bt\label{Th311}
Suppose that \eqref{Sig0} holds and $\mathbb{E}|\xi_0|^p + \mathbb{E}|\eta|^p < \infty$ for some $p \geq 1$.
\begin{enumerate}[(i)]
\item For any $t_0 \in [0, 1)$, there exists a solution $(X_t)_{t \in [0, t_0]}$ to the SDE \eqref{SDE0} starting from $\xi_0$, and
\begin{align}\label{Mo1}
\left\|\sup_{s \in [0, t]} \frac{|X_s|}{\sigma_s}\right\|_p \leq \|\xi_0\|_p + C_p \sqrt{\frac{\epsilon d}{\sigma_t}} + \|\eta\|_p \frac{\beta_t}{\sigma_t},
\end{align}
where $C_p > 0$ depends only on $p$. 
Moreover, for each $t \in (0, t_0]$, the law $\mu_t$ of $X_t$ has the density $\varphi_t(x)$, and for the constant $\mathcal{c}_p$ in \eqref{Cons0},
\begin{align}\label{WW2}
\mathcal{W}_p(\mu_t, \nu) \leq \sigma_t \|\xi_0\|_p + \mathcal{c}_p \sqrt{\epsilon \beta_t \sigma_t} + \sigma_t \|\eta\|_p.
\end{align}

\item Under each case of Lemma \ref{Le34}, strong uniqueness holds, and there exists a random variable $X_1 \sim \nu$ such that for  the constant $\mathcal{c}_p$ in \eqref{Cons0} and all $t \in [0, 1]$,
\begin{align}\label{Con6}
\|X_t - X_1\|_p \leq \sigma_t \|\xi_0 - \eta\|_p + 2 \mathcal{c}_p \sqrt{\epsilon \sigma_t}.
\end{align}
\end{enumerate}
\et

\begin{proof}
(i) Fix $t_0 \in [0, 1)$. As in the proof of Theorem \ref{Th23}, by \eqref{GG1}, we have
\begin{align*}
\int^{t_0}_0 \int_{\mathbb{R}^d} |b_t(x)| \varphi_t(x) \dif x \dif t 
&\leq \int^{t_0}_0 |(\log \sigma_t)'| \left( \mathbb{E}|\sigma_t \xi_0 + \bar{\xi}_t + \beta_t \eta| + \mathbb{E}|\eta| \right) \dif t < \infty.
\end{align*}
By the superposition principle (see \cite[Theorem 2.5]{Tr13}), there exists at least one solution to the SDE \eqref{SDE0} starting from $\xi_0$, such that for each $t \in (0, t_0]$,
$$
(\mathbb{P} \circ X^{-1}_t)(\dif x) = \mu_t(\dif x) = \varphi_t(x) \dif x.
$$
We now prove \eqref{Mo1}. By the chain rule and \eqref{SD2}, we have
\begin{align*}
\dif \left( \frac{X_t}{\sigma_t} \right) 
= \frac{\dif X_t}{\sigma_t} - \frac{\sigma'_t X_t}{\sigma_t^2} \dif t 
= \frac{b_t(X_t)}{\sigma_t} \dif t + \frac{\sqrt{\epsilon \beta'_t}}{\sigma_t} \dif W_t - \frac{\sigma'_t X_t}{\sigma_t^2} \dif t 
= -\frac{\sigma'_t \mathcal{D}_t(X_t)}{\sigma^2_t} \dif t + \frac{\sqrt{\epsilon \beta'_t}}{\sigma_t} \dif W_t.
\end{align*}
Since $\sigma'_t < 0$, we obtain
\begin{align*}
\left\|\sup_{s \in [0, t]} \frac{|X_s|}{\sigma_s}\right\|_p 
&\leq \|\xi_0\|_p - \int^t_0 \frac{\sigma'_s \|\mathcal{D}_s(X_s)\|_p}{\sigma_s^2} \dif s 
+ \left\|\sup_{s \in [0, t]} \left| \int^s_0 \frac{\sqrt{\epsilon \beta'_r}}{\sigma_r} \dif W_r \right| \right\|_p.
\end{align*}
By definition \eqref{DD8} and Jensen's inequality, we have
\begin{align*}
\|\mathcal{D}_s(X_s)\|_p 
&= \left( \int_{\mathbb{R}^d} \left| \frac{\mathbb{E}[\eta \rho_s(x - \beta_s \eta - \sigma_s \xi_0)]}{\mathbb{E} \rho_s(x - \beta_s \eta - \sigma_s \xi_0)} \right|^p \varphi_s(x) \dif x \right)^{1/p} \leq \|\eta\|_p.
\end{align*}
On the other hand, by the Burkholder-Davis-Gundy (BDG) inequality and \eqref{DS13}, we have
\begin{align*}
\mathbb{E} \left( \sup_{s \in [0, t]} \left| \int^s_0 \frac{\sqrt{\epsilon \beta'_r}}{\sigma_r} \dif W_r \right| \right)^p 
&\leq C_p \left( d \int^{t}_0 \frac{\epsilon \beta'_r}{\sigma^2_r} \dif r \right)^{p/2} 
= C_p \left( \frac{d \epsilon \beta_t}{\sigma_t} \right)^{p/2}.
\end{align*}
Combining these calculations, we obtain
\begin{align}\label{KJ12}
\left\|\sup_{s \in [0, t]} \frac{|X_s|}{\sigma_s}\right\|_p 
\leq \|\xi_0\|_p + \|\eta\|_p \frac{\beta_t}{\sigma_t} + C_p \sqrt{\frac{\epsilon d \beta_t}{\sigma_t}}.
\end{align}
This proves \eqref{Mo1}. The bound \eqref{WW2} follows directly from \eqref{WW1}.

(ii) Next, we prove strong uniqueness. By the Yamada-Watanabe theorem, it suffices to show pathwise uniqueness. Let $\widetilde{X}_t$ be another solution to the SDE \eqref{SDE0} starting from $\xi_0$. Then,
$$
|X_t - \widetilde{X}_t| \leq \int^t_0 |b_s(X_s) - b_s(\widetilde{X}_s)| \dif s.
$$
Fix $R > 0$ and define the stopping time
$$
\tau_R := \inf \left\{ t \in [0, t_0] : |X_t| > R \text{ or } |\widetilde{X}_t| > R \right\}.
$$
By Lemma \ref{Le34}, we have $\sup_{s \in [0, t_0]} \sup_{|x| \leq R} |\nabla b_s(x)| \leq C_R$. Thus,
$$
|X_{t \wedge \tau_R} - \widetilde{X}_{t \wedge \tau_R}| \leq \int^{t \wedge \tau_R}_0 |b_s(X_s) - b_s(\widetilde{X}_s)| \dif s 
\leq C_R \int^{t \wedge \tau_R}_0 |X_{s \wedge \tau_R} - \widetilde{X}_{s \wedge \tau_R}| \dif s.
$$
By Gronwall's inequality, we conclude that $|X_{t \wedge \tau_R} - \widetilde{X}_{t \wedge \tau_R}| = 0$ for all $t \in [0, t_0]$. 
Finally, since
$$
\mathbb{P}(\tau_R < t_0) = \mathbb{P}\left( \sup_{t \in [0, t_0]} (|X_t| + |\widetilde{X}_t|) \geq R \right) \to 0 \quad \text{as } R \to \infty,
$$
we have $\lim_{R \to \infty} \tau_R = t_0$. Thus, $X_t = \widetilde{X}_t$ for all $t \in [0, t_0]$.

(iii) Fix $0 \leq t_0 < t_1 < 1$. As in the proof of \eqref{SQ12} and by \eqref{SQ13}, we have
\begin{align*}
\left\| \int^{t_1}_{t_0} b_s(X_s) \dif s \right\|_p 
&\leq \int^{t_1}_{t_0} |(\log \sigma_s)'| \left( \int_{\mathbb{R}^d} \int_{\mathbb{R}^d} |x - y|^p \rho_s(x - \beta_s y) \nu(\dif y) \dif x \right)^{1/p} \dif s \\
&= \int^{t_1}_{t_0} |(\log \sigma_s)'| \left( \mathbb{E} \int_{\mathbb{R}^d} \int_{\mathbb{R}^d} |x - y|^p \phi_{\epsilon \beta_s \sigma_s}(x - \beta_s y - \sigma_s \xi_0) \nu(\dif y) \dif x \right)^{1/p} \dif s.
\end{align*}
By the change of variables $x - \beta_s y - \sigma_s \xi_0 = \sigma_s z$ and $\beta_s + \sigma_s = 1$, we obtain
\begin{align*}
\left\| \int^{t_1}_{t_0} b_s(X_s) \dif s \right\|_p 
&\leq \int^{t_1}_{t_0} |\sigma'_s| \left( \mathbb{E} \int_{\mathbb{R}^d} \int_{\mathbb{R}^d} |z + \xi_0 - y|^p \phi_{\epsilon \beta_s / \sigma_s}(z) \nu(\dif y) \dif z \right)^{1/p} \dif s \\
&= \int^{t_1}_{t_0} |\sigma'_s| \left\| \sqrt{\epsilon \beta_s / \sigma_s} \xi + \xi_0 - \eta \right\|_p \dif s,
\end{align*}
where $\xi \sim N(0, \mathbb{I}_d)$ is independent of $\xi_0$ and $\eta$. Hence,
\begin{align}
\|X_{t_1} - X_{t_0}\|_p 
&\leq \left\| \int^{t_1}_{t_0} b_s(X_s) \dif s \right\|_p + \left\| \int^{t_1}_{t_0} \sqrt{\epsilon \beta'_s} \dif W_s \right\|_p\no \\
&\leq \int^{t_1}_{t_0} |\sigma'_s| \left\| \sqrt{\epsilon \beta_s / \sigma_s} \xi + \xi_0 - \eta \right\|_p \dif s + \mathcal{c}_p \sqrt{\epsilon (\beta_{t_1} - \beta_0)}\no \\
&\leq \sqrt{\epsilon} \mathcal{c}_p \int^{t_1}_{t_0} \frac{|\sigma'_s|}{\sqrt{\sigma_s}} \dif s + \|\xi_0 - \eta\|_p (\sigma_{t_0} - \sigma_{t_1}) + \mathcal{c}_p \sqrt{\epsilon (\sigma_{t_0} - \sigma_{t_1})}\no \\
&=\sqrt{\epsilon} \mathcal{c}_p (\sqrt{\sigma_{t_0}} - \sqrt{\sigma_{t_1}}) + \|\xi_0 - \eta\|_p (\sigma_{t_0} - \sigma_{t_1}) + \mathcal{c}_p \sqrt{\epsilon (\sigma_{t_0} - \sigma_{t_1})}. \label{Con01}
\end{align}
This establishes the desired convergence.
\end{proof}

We now provide the proof of Theorem \ref{Th02}.

\begin{proof}[Proof of Theorem \ref{Th02}]
If we set $\beta_t = t$ and $\sigma_t = 1 - t$, then by \eqref{SD2} and \eqref{DD8}, we obtain \eqref{DD108}. Under the assumption $|\eta| \leq K$, by \eqref{KJ681}, the drift $b_t(x)$ is globally Lipschitz continuous, and (i) follows directly. Claims (ii) and (iii) follow from \cite[Theorem 1.2]{MPZ21}. For (iv), by \eqref{Con01}, we have
\begin{align*}
\lim_{t_0, t_1 \uparrow 1} \int_{\mathbb{R}^d} \mathbb{E} |X_{t_1}(x_0) - X_{t_0}(x_0)|^p \mu(\dif x_0) 
= \lim_{t_0, t_1 \uparrow 1} \|X_{t_1}(\xi_0) - X_{t_0}(\xi_0)\|^p_p = 0.
\end{align*}
By Fubini's theorem, there exists a random field $X_1(x_0, \omega)$ defined for $\mu$-almost all $x_0$ such that
$$
\lim_{t_0 \uparrow 1} \int_{\mathbb{R}^d} \mathbb{E} |X_{t_1}(x_0) - X_1(x_0)|^p \mu(\dif x_0) = 0.
$$
The conclusion (iv) now follows from \eqref{Con01}.
\end{proof}

In the following, we assume that the initial distribution $\xi_0$ follows a Gaussian distribution and derive several alternative expressions for $b_t(x)$. Depending on the regularity of the target distribution, we consider three cases:
\begin{enumerate}[(a)]
    \item \textbf{(Zero-order scheme)} We make no assumptions about the target distribution $\nu$.
    \item \textbf{(First-order scheme)} We assume that the target distribution $\nu$ admits a density $\rho$.
    \item \textbf{(Second-order scheme)} We assume that the target distribution $\nu$ has a $C^1$-density $\rho$.
\end{enumerate}

\bl[Zero-order scheme]\label{Le38}
Let $\gamma > 0$ and $\xi_0 \sim N(0, \gamma \mathbb{I}_d)$ be independent of $\eta$. Then,
\begin{align}\label{AA36}
    b_t(x) = \frac{\sigma'_t \mathbb{E}\left[(x - \eta)\e^{-|x - \beta_t \eta|^2 / (2 \ell_t \sigma_t)}\right]}{\sigma_t \mathbb{E} \left[\e^{-|x - \beta_t \eta|^2 / (2 \ell_t \sigma_t)}\right]}, \ \
    \varphi_t(x) = \frac{\mathbb{E} \left[\e^{-|x - \beta_t \eta|^2 / (2 \ell_t \sigma_t)}\right]}{(2 \pi \ell_t \sigma_t)^{d/2}},
\end{align}
where $\eps>0$ and $\ell_t := \epsilon \beta_t + \gamma \sigma_t$. Moreover, we also have
\begin{align}\label{AA30}
    b_t(x) = \ell_t (\log \beta_t)' \nabla_x\log \mathbb{E} \left[\e^{\beta_t (\sigma_t |\eta|^2 - |x - \eta|^2) / (2 \ell_t \sigma_t)}\right]
    = (\log \beta_t)' \left[\ell_t \nabla_x \log \varphi_t(x) + x\right].
\end{align}
\el

\begin{proof}
Let $\gamma > 0$ and $\xi_0 \sim N(0, \gamma \mathbb{I}_d)$. By \eqref{SQ13} and Lemma \ref{Le81}, we have
\begin{align}\label{DZ2}
    \rho_t(x) = \mathbb{E} \phi_{\epsilon \beta_t \sigma_t}(x - \sigma_t \xi_0) = \phi_{\epsilon \beta_t \sigma_t + \gamma \sigma_t^2}(x) = \phi_{\ell_t \sigma_t}(x).
\end{align}
Substituting this into the definition of $\mathcal{D}_t(x)$ and using the independence of $\xi_0$ and $\eta$, we get
$$
    \mathcal{D}_t(x) = \frac{\mathbb{E} \left[\eta \phi_{\ell_t \sigma_t}(x - \beta_t \eta)\right]}{\mathbb{E} \phi_{\ell_t \sigma_t}(x - \beta_t \eta)},
$$
and by the definition of $\varphi_t(x)$,
$$
    \varphi_t(x) = \mathbb{E} \rho_t(x - \beta_t \eta) = \mathbb{E} \phi_{\ell_t \sigma_t}(x - \beta_t \eta).
$$
Thus, we obtain \eqref{AA36} by \eqref{SD2}. Moreover, noting that
\begin{align}\label{AA306}
    -|x - \beta_t \eta|^2 = \sigma_t \beta_t |\eta|^2 - \beta_t |x - \eta|^2 - \sigma_t |x|^2,
\end{align}
by the definition of $b_t(x)$, we have
\begin{align}\label{AA396}
    b_t(x) &= \frac{\sigma'_t \mathbb{E} \left[(x - \eta)\e^{-|x - \beta_t \eta|^2 / (2 \ell_t \sigma_t)}\right]}{\sigma_t \mathbb{E} \left[\e^{-|x - \beta_t \eta|^2 / (2 \ell_t \sigma_t)}\right]} 
    = \frac{\sigma'_t \mathbb{E} \left[(x - \eta)\e^{\beta_t (\sigma_t |\eta|^2 -  |x - \eta|^2) / (2 \ell_t \sigma_t)}\right]}{\sigma_t \mathbb{E} \left[\e^{\beta_t (\sigma_t |\eta|^2 -  |x - \eta|^2) / (2 \ell_t \sigma_t)}\right]},
\end{align}
which yields the first expression in \eqref{AA30} by the chain rule and $\beta_t'=-\sigma_t'$. Moreover, 
by \eqref{AA36} and \eqref{AA306}, we clearly have
$$
    \mathbb{E} \left[\e^{\beta_t (\sigma_t |\eta|^2 - |x - \eta|^2) / (2 \ell_t \sigma_t)}\right] = (2 \pi \sigma_t \ell_t)^{d/2}\e^{|x|^2 / (2 \ell_t)} \varphi_t(x).
$$
Thus, we get the second expression in \eqref{AA30}.
\end{proof}

\br
In diffusion models, $\nabla_x \log \varphi_t(x)$ is called the score function. If $\epsilon = 0$, this reduces to the ODE;
and if $\gamma=0$, then
$$
b_t(x)=\eps\beta_t' \nabla_x\log \mathbb{E} \left[\e^{ |\eta|^2/(2\eps) - |x - \eta|^2/(2\eps \sigma_t)}\right].
$$
\er

Next, we derive a first-order expression for $b_t(x)$, which depends on the density of the target distribution.

\bl[First-order scheme]\label{Le39}
Let $\gamma,\lambda > 0$ and $\xi_0 \sim N(0, \gamma \mathbb{I}_d)$, and $\xi \sim N(0, \mathbb{I}_d / \lambda)$.
Suppose that $\eta \sim \rho_1(x) \dif x$. Then, for each $t \in (0, 1)$ and $x \in \mathbb{R}^d$,
\begin{align}
    b_t(x) &= \frac{\sqrt{\ell_t} \sigma'_t \mathbb{E} \left[\xi\e^{\beta_t |x - \sqrt{\ell_t \sigma_t} \xi|^2 / (2 \ell_t) + (\lambda - \beta_t) |\xi|^2 / 2} \rho_1(x - \sqrt{\ell_t \sigma_t} \xi)\right]}{\sqrt{\sigma_t} \mathbb{E} \left[\e^{\beta_t |x - \sqrt{\ell_t \sigma_t} \xi|^2 / (2 \ell_t) + (\lambda - \beta_t) |\xi|^2 / 2} \rho_1(x - \sqrt{\ell_t \sigma_t} \xi)\right]}, \label{AA3}
\end{align}
where $\eps>0$, $\ell_t := \epsilon \beta_t + \gamma \sigma_t$, and
\begin{align}\label{AA43}
    \varphi_t(x) = \frac{\mathbb{E} \left[\e^{\beta_t |x - \sqrt{\ell_t \sigma_t} \xi|^2 / (2 \ell_t) + (\lambda - \beta_t) |\xi|^2 / 2} \rho_1(x - \sqrt{\ell_t \sigma_t} \xi)\right]}{\lambda^{d/2}\e^{|x|^2 / (2 \ell_t)}}.
\end{align}
\el

\begin{proof}
Note that by the change of variable $x - y = z \sqrt{\ell_t \sigma_t}$,
\begin{align}
    &\mathbb{E} \left[\e^{(\beta_t \sigma_t |\eta|^2 - \beta_t |x - \eta|^2) / (2 \ell_t \sigma_t)}\right] = \int_{\mathbb{R}^d}\e^{(\beta_t \sigma_t |y|^2 - \beta_t |x - y|^2) / (2 \ell_t \sigma_t)} \rho_1(y) \dif y \nonumber \\
    &\qquad = (\ell_t \sigma_t)^{d/2} \int_{\mathbb{R}^d}\e^{\beta_t |x - \sqrt{\ell_t \sigma_t} z|^2 / (2 \ell_t) - \beta_t |z|^2 / 2} \rho_1(x - \sqrt{\ell_t \sigma_t} z) \dif z \nonumber \\
    &\qquad = (\ell_t \sigma_t)^{d/2} \int_{\mathbb{R}^d}\e^{\beta_t |x - \sqrt{\ell_t \sigma_t} z|^2 / (2 \ell_t) + (\lambda - \beta_t) |z|^2 / 2}\e^{-\lambda |z|^2 / 2} \rho_1(x - \sqrt{\ell_t \sigma_t} z) \dif z \nonumber \\
    &\qquad = (2 \pi \ell_t \sigma_t / \lambda)^{d/2} \mathbb{E} \left[\e^{\beta_t |x - \sqrt{\ell_t \sigma_t} \xi|^2 / (2 \ell_t) + (\lambda - \beta_t) |\xi|^2 / 2} \rho_1(x - \sqrt{\ell_t \sigma_t} \xi)\right], \label{HF1}
\end{align}
where $\xi \sim N(0, \mathbb{I}_d / \lambda)$. Similarly,
\begin{align*}
    &\mathbb{E} \left[(x - \eta)\e^{(\beta_t \sigma_t |\eta|^2 - \beta_t |x - \eta|^2) / (2 \ell_t \sigma_t)}\right] \\
    &\qquad = (2 \pi \ell_t \sigma_t / \lambda)^{d/2} \sqrt{\ell_t \sigma_t} \mathbb{E} \left[\xi\e^{\beta_t |x - \sqrt{\ell_t \sigma_t} \xi|^2 / (2 \ell_t) + (\lambda - \beta_t) |\xi|^2 / 2} \rho_1(x - \sqrt{\ell_t \sigma_t} \xi)\right].
\end{align*}
Substituting these into \eqref{AA396}, we obtain \eqref{AA3}. Moreover, by \eqref{AA306}, we also have
\begin{align*}
    \varphi_t(x) &= \frac{\mathbb{E} \left[\e^{-|x - \beta_t \eta|^2 / (2 \ell_t \sigma_t)}\right]}{(2 \pi \ell_t \sigma_t)^{d/2}} 
    = \frac{\e^{-|x|^2 / (2 \ell_t)} \mathbb{E} \left[\e^{(\beta_t \sigma_t |\eta|^2 - \beta_t |x - \eta|^2) / (2 \ell_t \sigma_t)}\right]}{(2 \pi \ell_t \sigma_t)^{d/2}},
\end{align*}
which, by \eqref{HF1}, gives \eqref{AA43}.
\end{proof}

\br
If $\lambda = \epsilon = \gamma = 1$, then $\ell_t = \epsilon \beta_t + \gamma \sigma_t = 1$, and \eqref{AA3} reduces to
$$
    b_t(x) = \frac{\sigma'_t \mathbb{E} \left[\xi\e^{(\beta_t |x - \sqrt{\sigma_t} \xi|^2 + \sigma_t |\xi|^2) / 2} \rho_1(x - \sqrt{\sigma_t} \xi)\right]}{\sqrt{\sigma_t} \mathbb{E} \left[\e^{(\beta_t |x - \sqrt{\sigma_t} \xi|^2 + \sigma_t |\xi|^2) / 2} \rho_1(x - \sqrt{\sigma_t} \xi)\right]}.
$$
\er

\br
By the change of variable $x - y = \sigma_t z$, we also have the alternative expression:
\begin{align}
    b_t(x) &= \frac{\sigma'_t \int_{\mathbb{R}^d} (x - y)\e^{-|x - \beta_t y|^2 / (2 \ell_t \sigma_t)} \rho_1(y) \dif y}{\sigma_t \int_{\mathbb{R}^d}\e^{-|x - \beta_t y|^2 / (2 \ell_t \sigma_t)} \rho_1(y) \dif y} 
    = \frac{\sigma'_t \int_{\mathbb{R}^d} z\e^{-\sigma_t |x + \beta_t z|^2 / (2 \ell_t)} \rho_1(x - \sigma_t z) \dif z}{\int_{\mathbb{R}^d}\e^{-\sigma_t |x + \beta_t z|^2 / (2 \ell_t)} \rho_1(x - \sigma_t z) \dif z}. \label{BTX1}
\end{align}
Although there is no essential difference between \eqref{BTX1} and \eqref{AA3}, the Monte Carlo method for calculating the integral strongly depends on the form of the above expression (see Section 4).
\er

\bl[Second-order scheme]\label{Le310}
Let $\gamma, \lambda > 0$ and $\xi_0 \sim N(0, \gamma \mathbb{I}_d)$, $\xi \sim N(0, \mathbb{I}_d / \sqrt{\lambda})$. Suppose that $\eta \sim \rho_1(x) \dif x$ with $\rho_1 \in C^1(\mathbb{R}^d)$. Then, for each $t \in (0, 1)$ and $x \in \mathbb{R}^d$,
\begin{align}\label{AA112}
    b_t(x) = (\log \beta_t)' \left(x + \frac{\ell_t \mathbb{E} \left[\e^{\beta_t |x - \sqrt{\ell_t \sigma_t} \xi|^2 / (2 \ell_t) + (\lambda - \beta_t) |\xi|^2 / 2} \nabla \rho_1(x - \sqrt{\ell_t \sigma_t} \xi)\right]}{\beta_t \mathbb{E} \left[\e^{\beta_t |x - \sqrt{\ell_t \sigma_t} \xi|^2 / (2 \ell_t) + (\lambda - \beta_t) |\xi|^2 / 2} \rho_1(x - \sqrt{\ell_t \sigma_t} \xi)\right]}\right),
\end{align}
where $\eps>0$ and $\ell_t := \epsilon \beta_t + \gamma \sigma_t$.
\el

\begin{proof}
By \eqref{AA30} and \eqref{HF1}, we have
$$
    b_t(x) = \ell_t (\log \beta_t)' \frac{\nabla_x \mathbb{E} \left[\e^{\beta_t |x - \sqrt{\ell_t \sigma_t} \xi|^2 / (2 \ell_t) + (\lambda - \beta_t) |\xi|^2 / 2} \rho_1(x - \sqrt{\ell_t \sigma_t} \xi)\right]}{\mathbb{E} \left[\e^{\beta_t |x - \sqrt{\ell_t \sigma_t} \xi|^2 / (2 \ell_t) + (\lambda - \beta_t) |\xi|^2 / 2} \rho_1(x - \sqrt{\ell_t \sigma_t} \xi)\right]}.
$$
Noting that
\begin{align*}
    &\nabla_x \mathbb{E} \left[\e^{\beta_t |x - \sqrt{\ell_t \sigma_t} \xi|^2 / (2 \ell_t) + (\lambda - \beta_t) |\xi|^2 / 2} \rho_1(x - \sqrt{\ell_t \sigma_t} \xi)\right] \\
    &\quad = \frac{\beta_t}{\ell_t} \mathbb{E} \left[(x - \sqrt{\ell_t \sigma_t} \xi)\e^{\beta_t |x - \sqrt{\ell_t \sigma_t} \xi|^2 / (2 \ell_t) + (\lambda - \beta_t) |\xi|^2 / 2} \rho_1(x - \sqrt{\ell_t \sigma_t} \xi)\right] \\
    &\qquad + \mathbb{E} \left[\e^{\beta_t |x - \sqrt{\ell_t \sigma_t} \xi|^2 / (2 \ell_t) + (\lambda - \beta_t) |\xi|^2 / 2} \nabla \rho_1(x - \sqrt{\ell_t \sigma_t} \xi)\right],
\end{align*}
by \eqref{AA3} and $\sigma_t = 1 - \beta_t$, we have
\begin{align*}
    b_t(x) = \beta_t' x + \sigma_t b_t(x) + \ell_t (\log \beta_t)' \frac{\mathbb{E} \left[\e^{\beta_t |x - \sqrt{\ell_t \sigma_t} \xi|^2 / (2 \ell_t) + (\lambda - \beta_t) |\xi|^2 / 2} \nabla \rho_1(x - \sqrt{\ell_t \sigma_t} \xi)\right]}{\mathbb{E} \left[\e^{\beta_t |x - \sqrt{\ell_t \sigma_t} \xi|^2 / (2 \ell_t) + (\lambda - \beta_t) |\xi|^2 / 2} \rho_1(x - \sqrt{\ell_t \sigma_t} \xi)\right]},
\end{align*}
which implies \eqref{AA112}. The proof is complete.
\end{proof}

\br
Note that $(\log \beta_t)'$ is not integrable near $0$. However, we can combine the first and second-order expressions to write
$$
    b_t(x) =
    \begin{cases}
        \frac{\sqrt{\ell_t} \sigma'_t \mathbb{E} \left[\xi\e^{\beta_t |x - \sqrt{\ell_t \sigma_t} \xi|^2 / (2 \ell_t) + (\lambda - \beta_t) |\xi|^2 / 2} \rho_1(x - \sqrt{\ell_t \sigma_t} \xi)\right]}{\sqrt{\sigma_t} \mathbb{E} \left[\e^{\beta_t |x - \sqrt{\ell_t \sigma_t} \xi|^2 / (2 \ell_t) + (\lambda - \beta_t) |\xi|^2 / 2} \rho_1(x - \sqrt{\ell_t \sigma_t} \xi)\right]}, & t \in [0, \tfrac{1}{2}), \\
        \frac{\beta_t'}{\beta_t} \left(x + \frac{\ell_t \mathbb{E} \left[\e^{\beta_t |x - \sqrt{\ell_t \sigma_t} \xi|^2 / (2 \ell_t) + (\lambda - \beta_t) |\xi|^2 / 2} \nabla \rho_1(x - \sqrt{\ell_t \sigma_t} \xi)\right]}{\beta_t \mathbb{E} \left[\e^{\beta_t |x - \sqrt{\ell_t \sigma_t} \xi|^2 / (2 \ell_t) + (\lambda - \beta_t) |\xi|^2 / 2} \rho_1(x - \sqrt{\ell_t \sigma_t} \xi)\right]}\right), & t \in [\tfrac{1}{2}, 1].
    \end{cases}
$$
In particular, if $0 \leq \beta_0', \beta_1' < \infty$, then
$$
    b_0(x) = \beta_0' (\mathbb{E} \eta - x), \quad b_1(x) = \beta_1' [x + \nabla \log \rho_1(x)].
$$
\er

\section{Particle Approximations of Stochastic Transport Maps}

Let $\eps, \gamma > 0$ and $\ell_t = \eps \beta_t + \gamma \sigma_t$, where 
$$
\beta'_t > 0, \quad \beta_0 = 0, \quad \beta_1 = 1, \quad \sigma_t + \beta_t = 1.
$$
Consider the following SDE:
\begin{align}\label{SDE90}
    \dif X_t = b_t(X_t) \dif t + \sqrt{\eps \beta'_t} \dif W_t, \quad X_0 = \xi_0 \sim N(0, \gamma \mathbb{I}_d),
\end{align}
with
\begin{align}\label{Ex0}
    b_t(x) := \frac{\sigma'_t \mathbb{E}[(x - \eta) \phi_{\eps \ell_t \sigma_t}(x - \beta_t \eta)]}{\sigma_t \mathbb{E} \phi_{\eps \ell_t \sigma_t}(x - \beta_t \eta)}
    = \frac{\sigma'_t \mathbb{E}[(x - \eta)\e^{-|x - \beta_t \eta|^2 / (2 \eps \ell_t \sigma_t)}]}{\sigma_t \mathbb{E}\e^{-|x - \beta_t \eta|^2 / (2 \eps \ell_t \sigma_t)}}.
\end{align}
It is natural to ask how to simulate the solution $X_t$. Along the time direction, we can use the classical Euler discretization. For the drift $b$, we need to calculate the expectations appearing in the denominator and numerator of $b_t(x)$. This will be done using Monte Carlo methods. We shall discuss the zero-order and first-order schemes and show the convergence rate for the corresponding particle approximations. For the zero-order scheme, when the target distribution has compact support, we can show strong convergence. For the first-order scheme, we show the convergence rate with respect to the total variation distance based on the entropy method.

First, we prepare the following crucial lemma.

\bl\label{Le41}
Let $\bxi^N := (\xi_1, \dots, \xi_N)$ be a sequence of i.i.d. random variables in $\mathbb{R}^d$ with common distribution $\mu$. Let $\xi \sim \mu$ be independent of $\bxi^N$, and let $\cC: \mathbb{R}^d \to [0, \infty)$ be a nonnegative measurable function. Let $p, q \in [1, \infty)$ and $r \in [1, \infty]$ satisfy $\frac{1}{p} = \frac{1}{q} + \frac{1}{r}$. Suppose that 
$$
    \|\cC(\xi)\|_q + \||\bxi|_*^N\|_r < \infty,
$$
where $|\bxi|_*^N := \sup_{i=1,\dots,N} |\xi_i|$. Then there is a constant $c_{p,q} > 0$ such that for all $N \in \mathbb{N}$,
$$
    \left\|\frac{\sum_{i=1}^N (\xi_i \cC(\xi_i))}{\sum_{i=1}^N \cC(\xi_i)} - \frac{\mathbb{E}(\xi \cC(\xi))}{\mathbb{E} \cC(\xi)}\right\|_p
    \leq \frac{c_{p,q} \||\bxi|_*^N\|_r \|\cC(\xi)\|_q}{N^{1 - (1/2) \vee (1/p)} \mathbb{E} \cC(\xi)}.
$$
\el

\begin{proof}
Observing that
\begin{align*}
    &\left|\frac{1}{N} \sum_{i=1}^N (\xi_i \cC(\xi_i)) \mathbb{E} \cC(\xi) - \frac{1}{N} \sum_{i=1}^N \cC(\xi_i) \mathbb{E}(\xi \cC(\xi))\right| \\
    &\quad \leq \left|\left(\frac{1}{N} \sum_{i=1}^N (\xi_i \cC(\xi_i)) - \mathbb{E}(\xi \cC(\xi))\right) \frac{1}{N} \sum_{i=1}^N \cC(\xi_i)\right| \\
    &\qquad + \left|\frac{1}{N} \sum_{i=1}^N (\xi_i \cC(\xi_i)) \left(\frac{1}{N} \sum_{i=1}^N \cC(\xi_i) - \mathbb{E} \cC(\xi)\right)\right|,
\end{align*}
we have
\begin{align*}
    \sI &:= \left|\frac{\sum_{i=1}^N (\xi_i \cC(\xi_i))}{\sum_{i=1}^N \cC(\xi_i)} - \frac{\mathbb{E}(\xi \cC(\xi))}{\mathbb{E} \cC(\xi)}\right| 
    = \left|\frac{\sum_{i=1}^N (\xi_i \cC(\xi_i)) \mathbb{E} \cC(\xi) - \sum_{i=1}^N \cC(\xi_i) \mathbb{E}(\xi \cC(\xi))}{\sum_{i=1}^N \cC(\xi_i) \mathbb{E} \cC(\xi)}\right| \\
    &\leq \frac{\left|\frac{1}{N} \sum_{i=1}^N [\xi_i \cC(\xi_i) - \mathbb{E}(\xi \cC(\xi))]\right|}{\mathbb{E} \cC(\xi)} + \frac{\sup_{i=1,\dots,N} |\xi_i| \left|\frac{1}{N} \sum_{i=1}^N [\cC(\xi_i) - \mathbb{E} \cC(\xi)]\right|}{\mathbb{E} \cC(\xi)}.
\end{align*}
Hence, by Hölder's inequality,
\begin{align*}
    \|\sI\|_p
    &\leq \frac{\left\|\sum_{i=1}^N (\xi_i \cC(\xi_i) - \mathbb{E}(\xi \cC(\xi))\right\|_p}{N \mathbb{E} \cC(\xi)}
    + \frac{\||\bxi|_*^N\|_r \left\|\sum_{i=1}^N (\cC(\xi_i) - \mathbb{E} \cC(\xi))\right\|_q}{N \mathbb{E} \cC(\xi)}.
\end{align*}
On the other hand, for independent random variables $(X_i)_{i=1,\dots,d}$ with zero mean, by Burkh\"older's inequality (see \cite[Theorem  6.3.6]{St11}),
\begin{align}\label{DA2}
    \mathbb{E}\left|\sum_{i=1}^N X_i\right|^p \leq C_p N^{(p/2) \vee 1 - 1} \sum_{i=1}^N \mathbb{E}|X_i|^p, \quad p \in [1, \infty).
\end{align}
Using this inequality, we derive that
$$
    \left\|\frac{\sum_{i=1}^N (\xi_i \cC(\xi_i))}{\sum_{i=1}^N \cC(\xi_i)} - \frac{\mathbb{E}(\xi \cC(\xi))}{\mathbb{E} \cC(\xi)}\right\|_p
    \leq \frac{c_p N^{(1/2) \vee (1/p)} \|\xi \cC(\xi)\|_p + c_r N^{(1/2) \vee (1/q)} \||\bxi|_*^N\|_r \|\cC(\xi)\|_q}{N \mathbb{E} \cC(\xi)},
$$
which yields the desired estimate by $\|\xi \cC(\xi)\|_p \leq \|\xi\|_r \|\cC(\xi)\|_q \leq \||\bxi|_*^N\|_r \|\cC(\xi)\|_q$.
\end{proof}

\br
If $p = 2$, then $\mathbb{E}\left|\sum_{i=1}^N X_i\right|^2 = \sum_{i=1}^N \mathbb{E}|X_i|^2$. In this case, we have
$$
    \left\|\frac{\sum_{i=1}^N (\xi_i \cC(\xi_i))}{\sum_{i=1}^N \cC(\xi_i)} - \frac{\mathbb{E}(\xi \cC(\xi))}{\mathbb{E} \cC(\xi)}\right\|_2
    \leq \frac{2 \||\bxi|_*^N\|_\infty \|\cC(\xi)\|_2}{\sqrt{N} \mathbb{E} \cC(\xi)}.
$$
\er

\subsection{SDE: Zero-Order Particle Approximation}
In this subsection, we construct a particle approximation for $b_t(x)$ for the zero-order scheme \eqref{AA36}. For this purpose, we introduce
$$
    \cC_t(x, y) :=\e^{-|x - \beta_t y|^2 / (2 \ell_t \sigma_t)},
$$
where $\ell_t := \eps \beta_t + \gamma \sigma_t$ and $\eps, \gamma > 0$. Then, by \eqref{AA36} and \eqref{AA30}, we have
\begin{align}\label{HS91}
    b_t(x) = \frac{\sigma'_t \mathbb{E}[(x - \eta) \cC_t(x, \eta)]}{\sigma_t \mathbb{E} \cC_t(x, \eta)}, \quad
    \varphi_t(x) = \frac{\mathbb{E} \cC_t(x, \eta)}{(2 \pi \sigma_t \ell_t)^{d/2}}.
\end{align}
Fix $N \in \mathbb{N}$ and $\by = (y_1, \dots, y_N) \in \mathbb{R}^{Nd}$. We also introduce
\begin{align}\label{HS11}
    b_t^{\by}(x) = \frac{\sigma'_t \sum_{i=1}^N (x - y_i) \cC_t(x, y_i)}{\sigma_t \sum_{i=1}^N \cC_t(x, y_i)}.
\end{align}
Consider the following SDE:
\begin{align}\label{HS1}
    \dif X^{\by}_t = b^{\by}_t(X^\by_t) \dif t + \sqrt{\eps \beta'_t} \dif W_t, \quad X^\by_0 \sim N(0, \gamma \mathbb{I}_d).
\end{align}
Noting that
$$
    \nabla_x b_t^\by(x) = \frac{\sigma'_t}{\sigma_t} \left[\mathbb{I} - \frac{\sum_i [(y_i \otimes y_i) \cC_t(x, y_i)]}{\ell_t \sigma_t \sum_i \cC_t(x, y_i)} + \frac{\left(\sum_i y_i \cC_t(x, y_i)\right)^{\otimes 2}}{\ell_t \sigma_t \left(\sum_i \cC_t(x, y_i)\right)^2}\right],
$$
by $\beta'_t = -\sigma'_t > 0$, we have for $|y_i| \leq K$,
\begin{align}\label{HG1}
    \<z, \nabla_x b_t^{\by}(x) z\> \leq \frac{\beta'_t}{\sigma_t} \left(\sup_{i=1,\dots,N} \frac{\<y_i, z\>^2}{\ell_t \sigma_t} - |z|^2\right)
    \leq \frac{\beta'_t}{\sigma_t} \left(\frac{K^2}{\ell_t \sigma_t} - 1\right) |z|^2.
\end{align}
Hence, SDE \eqref{HS1} admits a unique solution. For simplicity of notation, we write
\begin{align}\label{HS12}
    \cE^N_t(x, \by) := \frac{\mathbb{E}[\xi \cC_t(x, \eta)]}{\mathbb{E} \cC_t(x, \eta)} - \frac{\sum_{i=1}^N [y_i \cC_t(x, y_i)]}{\sum_{i=1}^N \cC_t(x, y_i)}.
\end{align}
We first prove the following lemma.

\bl\label{Le43}
Suppose that $|\eta| \leq K$ and $\mathbb{E} \eta \equiv 0$. Let $\boldsymbol{\eta} = (\eta_1, \dots, \eta_N)$ be $N$-independent copies of $\eta$. For any $p \geq 2$, there is a constant $C_p > 0$ such that for all $s \in (0, 1)$,
$$
    \mathbb{E}|\cE^N_s(X_s, \boldsymbol{\eta})|^p \leq C_p K^p N^{-p/2} \mathbb{E}\e^{(p-1) \beta_s^2 (p |\eta|^2 + \mathbb{E}|\eta|^2) / (2 \ell_s \sigma_s)},
$$
where $X_s$ is the solution of SDE \eqref{SDE0} and $C_2 = 4$.
\el

\begin{proof}
Since $X_s$ and $\boldsymbol{\eta}$ are independent, by Lemma \ref{Le41} and \eqref{HS91}, we have
\begin{align*}
    \mathbb{E}|\cE^N_s(X_s, \boldsymbol{\eta})|^p
    &= \int_{\mathbb{R}^d} \mathbb{E}|\cE^N_s(x, \boldsymbol{\eta})|^p \varphi_s(x) \dif x
    \leq \int_{\mathbb{R}^d} \frac{C_p K^p \mathbb{E}|\cC_s(x, \eta)|^p}{N^{p/2} (\mathbb{E} \cC_s(x, \eta))^p} \varphi_s(x) \dif x \\
    &= \frac{C_p K^p}{N^{p/2} (2 \pi \sigma_s \ell_s)^{d/2}}
    \int_{\mathbb{R}^d} \frac{\mathbb{E}|\cC_s(x, \eta)|^p}{(\mathbb{E} \cC_s(x, \eta))^{p-1}} \dif x,
\end{align*}
where $C_2 = 4$. Noting that by Jensen's inequality,
$$
    \mathbb{E} \cC_s(x, \eta) \geq\e^{-\mathbb{E} |x - \beta_s \eta|^2 / (2 \ell_s \sigma_s)},
$$
we further have
\begin{align*}
    \mathbb{E}|\cE^N_s(X_s, \boldsymbol{\eta})|^p
    &\leq \frac{C_p K^p}{N^{p/2} (2 \pi \sigma_s \ell_s)^{d/2}} \int_{\mathbb{R}^d}
    \mathbb{E}\e^{((p-1) \mathbb{E} |x - \beta_s \eta|^2 - p |x - \beta_s \eta|^2) / (2 \ell_s \sigma_s)} \dif x.
\end{align*}
Noting that by $\mathbb{E} \eta = 0$,
$$
    \frac{(p-1) \mathbb{E} |x - \beta_s \eta|^2 - p |x - \beta_s \eta|^2}{2 \ell_s \sigma_s}
    = \frac{(p-1) \beta_s^2 (p |\eta|^2 + \mathbb{E}|\eta|^2) - |x - p \beta_s \eta|^2}{2 \ell_s \sigma_s},
$$
by Fubini's theorem, we get
$$
    \mathbb{E}|\cE^N_s(X_s, \boldsymbol{\eta})|^p
    \leq C_p K^p N^{-p/2} \mathbb{E}\e^{(p-1) \beta_s^2 (p |\eta|^2 + \mathbb{E}|\eta|^2) / (2 \ell_s \sigma_s)}.
$$
The proof is complete.
\end{proof}

Now we can show the following convergence rate.

\bt\label{Th45}
Suppose that $|\eta| \leq K$ and $\mathbb{E} \eta \equiv 0$. Let $\boldsymbol{\eta} = (\eta_1, \dots, \eta_N)$ be $N$-independent copies of $\eta$ and $\mu^\by_t(\dif x)$ be the law of $X^\by_t$. Then for all $N \in \mathbb{N}$,
$$
\mathbb{E}\Big(\|\mu^\by_t - \mu_t\|_{\var}^2|_{\by = \boldsymbol{\eta}}\Big) \leq \frac{8 K^2}{N} \mathbb{E}\left(\frac{\e^{\beta_t (2 |\eta|^2 + \mathbb{E}|\eta|^2) / (2 \eps \sigma_t)}}{2 |\eta|^2 + \mathbb{E}|\eta|^2}\right), \quad \forall t \in (0, 1).
$$
\et

\begin{proof}
By the classical Csiszár–Kullback–Pinsker (CKP) inequality (see \cite[(22.25)]{Villani2009})
and the entropy formula (see Lemma \ref{Lem34}), we have
\begin{align*}
    \|\mu^\by_t - \mu_t\|_{\var}^2 &\leq 2 \cH(\mu_t | \mu^\by_t)
    = \int^t_0 \frac{\mathbb{E}|b^\by_s(X_s) - b_s(X_s)|^2}{\eps \beta'_s} \dif s.
\end{align*}
Note that by \eqref{HS11}, \eqref{HS12}, and $\sigma_t = 1 - \beta_t$,
$$
    b^\by_t(x) - b_t(x) = \sigma_t' \cE^N_t(x, \by) / \sigma_t = -\beta_t' \cE^N_t(x, \by) / \sigma_t.
$$
Thus, by Lemma \ref{Le43} with $p = 2$, we get
\begin{align*}
    &\mathbb{E}\Big(\|\mu^\by_t - \mu_t\|_{\var}^2|_{\by = \boldsymbol{\eta}}\Big)
    \leq \int^t_0 \frac{\beta'_s}{\eps \sigma_s^2} \mathbb{E}|\cE^N_s(X_s, \boldsymbol{\eta})|^2 \dif s \leq \frac{4 K^2}{\eps N} \int^t_0 \frac{\beta'_s}{\sigma_s^2}
    \mathbb{E}\e^{\beta_s^2 (2 |\eta|^2 + \mathbb{E}|\eta|^2) / (2 \ell_s \sigma_s)} \dif s.
\end{align*}
Noting that for $\kappa = |\eta|^2 + \mathbb{E}|\eta|^2 / 2$,
\begin{align*}
    \int^t_0 \frac{\beta'_s}{\sigma_s^2}
   \e^{\kappa \beta_s^2 / ((\eps \beta_s + \gamma \sigma_s) \sigma_s)} \dif s
    &\leq \int^t_0\e^{\kappa \beta_s / (\eps \sigma_s)} \dif \sigma^{-1}_s
    = \int^{\sigma^{-1}_t}_1\e^{\kappa (x - 1) / \eps} \dif x
    = \frac{\eps}{\kappa} (e^{\kappa \beta_t / (\eps \sigma_t)} - 1),
\end{align*}
we further have
\begin{align*}
    \mathbb{E}\Big(\|\mu^\by_t - \mu_t\|_{\var}^2|_{\by = \boldsymbol{\eta}}\Big)
    \leq \frac{8 K^2}{N} \mathbb{E}\left(\frac{\e^{\beta_t (2 |\eta|^2 + \mathbb{E}|\eta|^2) / (2 \eps \sigma_t)}}{2 |\eta|^2 + \mathbb{E}|\eta|^2}\right).
\end{align*}
The proof is complete.
\end{proof}

Let $\boldsymbol{\eta} = (\eta_1, \dots, \eta_N)$ be $N$-independent copies of $\eta$, and define
$$
    b_t^N(x) = \frac{\sigma'_t \sum_{i=1}^N (x - \eta_i) \cC_t(x, \eta_i)}{\sigma_t \sum_{i=1}^N \cC_t(x, \eta_i)}.
$$
We also consider the following SDE:
$$
    \dif X^N_t = b^N_t(X^N_t) \dif t + \sqrt{\eps \beta'_t} \dif W_t, \quad X^N_0 = \xi_0 \sim N(0, \gamma \mathbb{I}_d).
$$
Clearly, we have
$$
    X^N_t = X^\by_t|_{\by = \boldsymbol{\eta}}.
$$
Let $\mu^N_t$ be the law of $X^N_t$. Under the assumptions of Theorem \ref{Th45}, we have
$$
    \|\mu^N_t - \mu_t\|_{\var}^2 \leq \frac{8 K^2}{N} \mathbb{E}\left(\frac{\e^{\beta_t (2 |\eta|^2 + \mathbb{E}|\eta|^2) / (2 \eps \sigma_t)}}{2 |\eta|^2 + \mathbb{E}|\eta|^2}\right).
$$
In fact, we also have the following strong convergence result.

\bt
Suppose that $|\eta| \leq K$, $\mathbb{E} \eta \equiv 0$, and $0 \leq \eps \leq \gamma$. For any $p \geq 2$, there is a constant $C_p > 0$ such that for all $t \in (0, 1)$ and $N \in \mathbb{N}$,
$$
    \|X^N_t - X_t\|_p \leq C_p \sqrt{\eps}\e^{K^2 \beta_t / (2 \eps \sigma_t)} \Big(\mathbb{E}\e^{(p-1) (p |\eta|^2 + \mathbb{E}|\eta|^2) / (2 \eps \sigma_t)}\Big)^{1/p} / \sqrt{N}.
$$
\et

\begin{proof}
By the chain rule, \eqref{HG1}, and $ab \leq a^2 + b^2 / 4$, we have
\begin{align*}
    \frac{\dif |X^N_t - X_t|^2}{\dif t} &= \<X^N_t - X_t, b^N_t(X^N_t) - b^N_t(X_t)\> + \<X^N_t - X_t, b^N_t(X_t) - b_t(X_t)\> \\
    &\leq \frac{\beta'_t}{\sigma_t} \left(\frac{K^2}{\ell_t \sigma_t} - 1\right) |X^N_t - X_t|^2 + |X^N_t - X_t| |(\log \sigma_t)'| |\cE^N_t(X_t, \boldsymbol{\eta})| \\
    &\leq \frac{K^2 \beta'_t}{\ell_t \sigma^2_t} |X^N_t - X_t|^2 + \frac{\beta_t'}{4 \sigma_t} |\cE^N_t(X_t, \boldsymbol{\eta})|^2.
\end{align*}
Hence, by Gronwall's inequality,
$$
    |X^N_t - X_t|^2 \leq \int^t_0\e^{\int^t_s \frac{K^2 \beta'_r}{\ell_r \sigma^2_r} \dif r} \frac{\beta_s'}{4 \sigma_s} |\cE^N_s(X_s, \boldsymbol{\eta})|^2 \dif s.
$$
Since $0 \leq \eps \leq \gamma$, we have $\ell_s = \eps \beta_s + \gamma \sigma_s \geq \eps$ and
$$
    \int^t_s \frac{K^2 \beta'_r}{\ell_r \sigma^2_r} \dif r \leq \frac{K^2}{\eps} \int^t_s \frac{\beta'_r}{\sigma^2_r} \dif r
    = \frac{K^2}{\eps} \left(\sigma^{-1}_t - \sigma^{-1}_s\right).
$$
For $p \geq 2$, by Minkowski's inequality and Lemma \ref{Le43}, we obtain
\begin{align*}
    \|X^N_t - X_t\|_p^2 &= \||X^N_t - X_t|^2\|_{p/2} \leq\e^{K^2 / (\eps \sigma_t)}
    \int^t_0\e^{-K^2 / (\eps \sigma_s)} \frac{\beta_s'}{\sigma_s} \|\cE^N_s(X_s, \boldsymbol{\eta})\|_p^2 \dif s \\
    &\leq C_p\e^{K^2 / (\eps \sigma_t)} \int^t_0\e^{-K^2 / (\eps \sigma_s)} \frac{K^2 \beta_s'}{N \sigma_s}
    \Big(\mathbb{E}\e^{(p-1) \beta_s^2 (p |\eta|^2 + \mathbb{E}|\eta|^2) / (2 \ell_s \sigma_s)}\Big)^{2/p} \dif s.
\end{align*}
Since $\ell_s \geq \eps$, we further have
\begin{align*}
    \|X^N_t - X_t\|_p^2
    &\leq \frac{C_p K^2\e^{K^2 / (\eps \sigma_t)}}{N} \int^t_0\e^{-K^2 / (\eps \sigma_s)} \frac{\beta_s'}{\sigma_s}
    \dif s \Big(\mathbb{E}\e^{(p-1) (p |\eta|^2 + \mathbb{E}|\eta|^2) / (2 \eps \sigma_t)}\Big)^{2/p}.
\end{align*}
Note that by the change of variable,
\begin{align*}
    \int^t_0\e^{-K^2 / (\eps \sigma_s)} \frac{\beta_s'}{\sigma_s} \dif s
    = \int^{\sigma_t^{-1}}_1 \frac{\e^{-K^2 x / \eps}}{x} \dif x
    \leq \int^{\sigma_t^{-1}}_1\e^{-K^2 x / \eps} \dif x = \frac{\eps}{K^2} (e^{-K^2 / \eps} -\e^{-K^2 / (\eps \sigma_t)}).
\end{align*}
Hence,
\begin{align*}
    \|X^N_t - X_t\|_p^2
    &\leq \frac{C_p \eps\e^{K^2 \beta_t / (\eps \sigma_t)}}{N} \Big(\mathbb{E}\e^{(p-1) (p |\eta|^2 + \mathbb{E}|\eta|^2) / (2 \eps \sigma_t)}\Big)^{2/p}.
\end{align*}
The proof is complete.
\end{proof}

\subsection{SDE: First-Order Particle Approximation}
In this subsection, we construct a particle approximation for $b_t(x)$ for the first-order scheme \eqref{AA3}. Fix $\lambda, \eps, \gamma > 0$ and define
$$
\cC_t(x, y) =\e^{\beta_t |x - \sqrt{\ell_t \sigma_t} y|^2 / (2 \ell_t) + (\lambda - \beta_t) |y|^2 / 2} \rho_1(x - \sqrt{\ell_t \sigma_t} y),
$$
where $\ell_t = \eps \beta_t + \gamma \sigma_t$. Let $\xi \sim N(0, \mathbb{I}_d / \lambda)$. 
By \eqref{AA3} and \eqref{AA43}, we have
\begin{align}\label{AA35}
    b_t(x) &= \frac{\sqrt{\ell_t} \sigma'_t \mathbb{E}[\xi \cC_t(x, \xi)]}{\sqrt{\sigma_t} \mathbb{E} \cC_t(x, \xi)},\ 
    \varphi_t(x) = \frac{\mathbb{E} \cC_t(x, \xi)}{\lambda^{d/2}\e^{|x|^2 / (2 \ell_t)}}.
\end{align}
Let $\by = (y_1, \dots, y_N) \in \mathbb{R}^{Nd}$ be fixed. For $t \in [0, 1)$ and $x \in \mathbb{R}^d$, we introduce
$$
b^\by_t(x) = \frac{\sqrt{\ell_t} \sigma'_t \sum_{i=1}^N [y_i \cC_t(x, y_i)]}{\sqrt{\sigma_t} \sum_{i=1}^N \cC_t(x, y_i)}.
$$
Now we consider the following SDE:
$$
\dif X^\by_t = b^\by_t(X^\by_t) \dif t + \sqrt{\eps \beta'_t} \dif W_t, \quad X^\by_0 \sim N(0, \gamma \mathbb{I}_d).
$$
Clearly,
$$
|b^\by_t(x)| \leq \sqrt{\ell_t / \sigma_t} \beta'_t \sup_i |y_i|.
$$
Suppose that for any $t_0 \in (0, 1)$,
$$
\eps > 0, \quad \inf_{t \in [0, t_0]} \beta'_t > 0.
$$
Since $b^\by$ is bounded measurable, by the well-known results in the theory of SDEs (see \cite{XXZZ20}), the above SDE admits a unique strong solution. We first prepare the following lemma.

\bl
Let $\xi \sim N(0, \mathbb{I}_d / \lambda)$. For any $q \geq 1$, it holds that
\begin{align}\label{BC1}
    \mathbb{E} \cC^q_t(x, \xi)
    &= \int_{\mathbb{R}^d}
    \frac{\e^{(q \beta_t \sigma_t |z|^2 + (q (\lambda - \beta_t) - \lambda) |x - z|^2) / (2 \ell_t \sigma_t)}}{(2 \pi \ell_t \sigma_t / \lambda)^{d/2}} \rho_1^q(z) \dif z,
\end{align}
and for $\eta \sim \rho_1(x) \dif x$ with $\mathbb{E} \eta = 0$,
\begin{align}\label{BC2}
    \mathbb{E} \cC_t(x, \xi) \geq \frac{\e^{-\beta_t (|x|^2 + \beta_t \mathbb{E}|\eta|^2) / (2 \ell_t \sigma_t)}}{(2 \pi \ell_t \sigma_t / \lambda)^{d/2}}.
\end{align}
\el

\begin{proof}
By definition and the change of variable $x - \sqrt{\ell_t \sigma_t} y = z$, we have
\begin{align*}
    \mathbb{E} \cC^q_t(x, \xi)
    &= (2 \pi / \lambda)^{-d/2} \int_{\mathbb{R}^d}
   \e^{q \beta_t |x - \sqrt{\ell_t \sigma_t} y|^2 / (2 \ell_t)} \rho_1^q(x - \sqrt{\ell_t \sigma_t} y)\e^{(q (\lambda - \beta_t) - \lambda) |y|^2 / 2} \dif y \\
    &= (2 \pi \ell_t \sigma_t / \lambda)^{-d/2} \int_{\mathbb{R}^d}\e^{q \beta_t |z|^2 / (2 \ell_t)}
    \rho_1^q(z)\e^{(q (\lambda - \beta_t) - \lambda) |x - z|^2 / (2 \ell_t \sigma_t)} \dif z,
\end{align*}
which gives \eqref{BC1}. On the other hand, for $q = 1$, by Jensen's inequality, we have
$$
    \mathbb{E} \cC_t(x, \xi) = \frac{\mathbb{E}\e^{\beta_t |\eta|^2 / (2 \ell_t) - \beta_t |x - \eta|^2 / (2 \ell_t \sigma_t)}}{(2 \pi \ell_t \sigma_t / \lambda)^{d/2}}
    \geq \frac{\e^{\beta_t \mathbb{E}|\eta|^2 / (2 \ell_t) - \beta_t \mathbb{E}|x - \eta|^2 / (2 \ell_t \sigma_t)}}{(2 \pi \ell_t \sigma_t / \lambda)^{d/2}},
$$
which gives \eqref{BC2} by $\beta_t + \sigma_t = 1$ and $\mathbb{E} \eta = 0$. The proof is complete.
\end{proof}

Now we can show the following convergence rate.

\bt\label{Th47}
Let $0 < \eps \leq \gamma < \infty$. Let $\mu^\by_t(\dif x)$ be the law of $X^\by_t$ and $\mu^N_t$ be the law of $X^\by_t|_{\by = \bxi}$. For $q > 2$ and $\lambda \in (0, 1)$ with $\delta := q - (3q - 4) \lambda > 0$, there is a constant $C_{q,d} > 0$ such that for each $N \in \mathbb{N}$ and $\bxi = (\xi_1, \dots, \xi_N)$ being i.i.d. random variables with common distribution $N(0, \mathbb{I}_d / \lambda)$,
$$
    \|\mu^N_t - \mu_t\|_{\var}^2 \leq
    \mathbb{E}\Big(\|\mu^\by_t - \mu_t\|_{\var}^2|_{\by = \bxi}\Big) \leq 
    C_{q,d} \left(\frac{(\ell_t \sigma_t)^{d(q-1)/q} \beta_t \sV_t^{2/q}}{\eps \sigma_t \delta^{d/q} \lambda^{(3/2 - 2/q) d}}\right) \frac{\ln N}{N}, \quad t \in (0, 1),
$$
where $\ell_t = \eps \beta_t + \gamma \sigma_t$ and
$$
    \sV_t := \int_{\mathbb{R}^d}\e^{q^2 |z|^2 / \ell_t}\e^{q \mathbb{E}|\eta|^2 / (4 \ell_t \sigma_t)} \rho_1^q(z) \dif z.
$$
\et

\begin{proof}
Following the proof of Theorem \ref{Th45}, we have
\begin{align}\label{BG9}
    \mathbb{E}\Big(\|\mu^\by_t - \mu_t\|_{\var}^2|_{\by = \bxi}\Big)
    &\leq \int^t_0 \frac{\beta'_s}{\eps \sigma_s^2} \int_{\mathbb{R}^d} \mathbb{E}|\cE^N_s(x, \bxi)|^2 \varphi_s(x) \dif x \dif s,
\end{align}
where $\varphi_s(x)$ is given by \eqref{BC1} and
$$
    \cE^N_s(x, \bxi) := \frac{\mathbb{E}[\xi \cC_s(x, \xi)]}{\mathbb{E} \cC_s(x, \xi)} - \frac{\sum_{i=1}^N [\xi_i \cC_s(x, \xi_i)]}{\sum_{i=1}^N \cC_s(x, \xi_i)}.
$$
Now, applying Lemma \ref{Le41} with $p = 2$ and $q, r \in (2, \infty)$ with $\frac{1}{2} = \frac{1}{q} + \frac{1}{r}$, and by \eqref{AA35}, we obtain
\begin{align*}
    \mathbb{E}|\cE^N_s(x, \bxi)|^2
    &\lesssim \frac{\||\bxi|_*^N\|^2_r \|\cC_s(x, \xi)\|^2_q}{N (\mathbb{E} \cC_s(x, \xi))^2}
    \lesssim \frac{\ln N \|\cC_s(x, \xi)\|^2_q}{\lambda N (\mathbb{E} \cC_s(x, \xi))^2},
\end{align*}
where the second inequality is due to Lemma \ref{Le82} below. Hence, by \eqref{AA35} and \eqref{BC2},
\begin{align}\label{BC4}
    \int_{\mathbb{R}^d} \mathbb{E}|\cE^N_s(x, \bxi)|^2 \varphi_s(x) \dif x
    &\lesssim \frac{\ln N}{\lambda N} \int_{\mathbb{R}^d} \frac{\|\cC_s(x, \xi)\|^2_q}{(\mathbb{E} \cC_s(x, \xi))^2} \varphi_s(x) \dif x
    \lesssim \frac{\ln N}{\lambda^{d/2} N} \int_{\mathbb{R}^d} \frac{\|\cC_s(x, \xi)\|^2_q}{\e^{h_s(x)}} \dif x,
\end{align}
where
\begin{align*}
    h_s(x) &= (\sigma_s |x|^2 - \beta_s (|x|^2 + \beta_s \mathbb{E}|\eta|^2)) / (2 \ell_s \sigma_s).
\end{align*}
Recalling the Gaussian density $\phi_\sigma(x)$, for $q > 2$, by Jensen's inequality and \eqref{BC1}, we have
\begin{align*}
    \sI_s &:= \left(\int_{\mathbb{R}^d} \|\cC_s(x, \xi)\|^2_q\e^{h_s(x)} \dif x\right)^{q/2}
    \leq \int_{\mathbb{R}^d} \frac{\mathbb{E}|\cC_s(x, \xi)|^q}{\e^{q h_s(x) / 2}} \phi_{\ell_s \sigma_s / \lambda}^{1 - q/2}(x) \dif x \\
    &= \int_{\mathbb{R}^d} \int_{\mathbb{R}^d}
    \frac{\e^{(q \beta_s \sigma_s |z|^2 + a_s |x - z|^2) / (2 \ell_s \sigma_s)
    + (q - 2) \lambda |x|^2 / (4 \ell_s \sigma_s)}}{(2 \pi \ell_s \sigma_s / \lambda)^{(2 - q) d/2}\e^{q h_s(x) / 2}} \rho_1^q(z) \dif z \dif x \\
    &= \int_{\mathbb{R}^d} \frac{\e^{q \beta_s |z|^2 / (2 \ell_s)} \rho_1^q(z)}{(2 \pi \ell_s \sigma_s / \lambda)^{(2 - q) d/2}} \left(\int_{\mathbb{R}^d}
   \e^{A_s(x, z) / (4 \ell_s \sigma_s)} \dif x\right) \dif z,
\end{align*}
where
$$
    a_s := q (\lambda - \beta_s) - \lambda = (q - 1) \lambda - q \beta_s,
$$
and
$$
    A_s(x, z) = 2 a_s |x - z|^2 + q \beta_s (|x|^2 + \beta_s \mathbb{E}|\eta|^2) - q \sigma_s |x|^2 + (q - 2) \lambda |x|^2.
$$
By elementary calculations, we have
\begin{align*}
    A_s(x, z) = -\delta \left|x - 2 a_s z / \delta\right|^2 + 4 |a_s z|^2 / \delta + 2 a_s |z|^2 + q \beta^2_s \mathbb{E}|\eta|^2,
\end{align*}
where $\delta := q - (3q - 4) \lambda > 0$. Hence,
\begin{align*}
    \sI_s &\leq \int_{\mathbb{R}^d} \frac{\e^{q \beta_s |z|^2 / (2 \ell_s)} \rho_1^q(z)}{(2 \pi \ell_s \sigma_s / \lambda)^{(2 - q) d/2}}
    \frac{\e^{(4 a_s^2 |z|^2 / \delta + 2 a_s |z|^2 + q \beta^2_s \mathbb{E}|\eta|^2) / (4 \ell_s \sigma_s)}}{(4 \pi \ell_s \sigma_s / \delta)^{-d/2}} \dif z \\
    &= \left(\frac{(2 \pi \ell_s \sigma_s)^{q - 1}}{\delta \lambda^{q - 2} / 2}\right)^{d/2} \int_{\mathbb{R}^d}
   \e^{g_s |z|^2 / (2 \ell_s \sigma_s)}\e^{q \beta^2_s \mathbb{E}|\eta|^2 / (4 \ell_s \sigma_s)} \rho_1^q(z) \dif z,
\end{align*}
where
$$
    g_s := q \beta_s \sigma_s + 2 a_s^2 / \delta + a_s.
$$
Substituting this into \eqref{BC4}, we obtain
\begin{align*}
    \int_{\mathbb{R}^d} \mathbb{E}|\cE^N_s(x, \bxi)|^2 \varphi_s(x) \dif x
    &\lesssim \frac{\ln N (\ell_s \sigma_s)^{d(q - 1) / q}}{N \lambda^{d/2} (\delta \lambda^{q - 2})^{d/q}} \left(\int_{\mathbb{R}^d}
   \e^{g_s |z|^2 / (2 \ell_s \sigma_s)}\e^{q \beta^2_s \mathbb{E}|\eta|^2 / (4 \ell_s \sigma_s)} \rho_1^q(z) \dif z\right)^{2/q}.
\end{align*}
Noting that 
$$
    a_s \leq q - q \beta_s = q \sigma_s,
$$
we have
\begin{align*}
    g_s &= q \beta_s \sigma_s + a_s^2 + a_s \leq 2 q \sigma_s + q^2 \sigma_s^2 \leq 2 q^2 \sigma_s.
\end{align*}
Hence, for $s \in [0, t]$,
\begin{align*}
    \int_{\mathbb{R}^d} \mathbb{E}|\cE^N_s(x, \bxi)|^2 \varphi_s(x) \dif x
    &\lesssim \frac{\ln N (\ell_s \sigma_s)^{d(q - 1) / q}}{N \delta^{d/q} \lambda^{(3/2 - 2/q) d}} \left(\int_{\mathbb{R}^d}
   \e^{q^2 |z|^2 / \ell_s}\e^{q \mathbb{E}|\eta|^2 / (4 \ell_s \sigma_s)} \rho_1^q(z) \dif z\right)^{2/q},
\end{align*}
which together with \eqref{BG9} yields the desired estimate.
\end{proof}

\section{Application to Sampling from Known Distributions}

In this section, we use the SDE \eqref{SDE90} to sample two typical distributions: high-dimensional anisotropic funnel distributions and Gaussian mixture distributions.

\subsection{High-Dimensional Anisotropic Funnel Distribution}

In this subsection, we sample the high-dimensional anisotropic funnel distribution using the first-order scheme. More precisely, for $x = (x_1, x^*_1) \in \mathbb{R} \times \mathbb{R}^{d-1}$, the density is given by
$$
\rho_1(x) = \frac{1}{\sqrt{2\pi}} \exp\left(-\frac{x_1^2}{2}\right) \left(\frac{1}{\sqrt{2\pi\e^{2\alpha x_1}}}\right)^{d-1} \exp\left(-\frac{|x_1^*|_2^2}{2e^{2\alpha x_1}}\right) = \phi_1(x_1) \phi_{\e^{2\alpha x_1}}(x_1^*),
$$
where the parameter $\alpha > 0$ controls the shape of the funnel, and
$$
\phi_\sigma(x) := (2\pi\sigma)^{-d/2}\e^{-|x|^2/(2\sigma)}.
$$
Now, consider the following first-order particle approximation SDE:
\begin{align}\label{SDE901}
    \dif X^N_t = b^N_t(X^N_t) \dif t + \sqrt{\eps \beta_t'} \dif W_t, \quad X^N_0 \sim N(0, \gamma \mI_d),
\end{align}
where $\eps, \gamma, \lambda > 0$, and for $\ell_t = \eps \beta_t + \gamma \sigma_t$ 
and $\xi_i \sim N(0, \mI_d / \lambda)$ being i.i.d. random variables,
$$
    b_t^N(x) = \frac{\sqrt{\ell_t} \sigma'_t \sum_{i=1}^N \left[ \xi_i \exp\left( \frac{\beta_t |x - \sqrt{\ell_t \sigma_t} \xi_i|^2}{2 \ell_t} + \frac{(\lambda - \beta_t) |\xi_i|^2}{2} \right) \rho_1(x - \sqrt{\ell_t \sigma_t} \xi_i) \right]}{\sqrt{\sigma_t} \sum_{i=1}^N \exp\left( \frac{\beta_t |x - \sqrt{\ell_t \sigma_t} \xi_i|^2}{2 \ell_t} + \frac{(\lambda - \beta_t) |\xi_i|^2}{2} \right) \rho_1(x - \sqrt{\ell_t \sigma_t} \xi_i)}.
$$
The following is the pseudocode of the algorithm based on the above SDE and Euler's scheme.
\begin{algorithm}[H]
\caption{Sampling the anisotropic funnel distribution $f$ via the first order scheme of SDEs}
\label{alg:sde_simulation}
\begin{algorithmic}[1]
\STATE \textbf{Initialize Parameters:}
\STATE $D$ \COMMENT{Spatial dimension}; $N$ \COMMENT{Number of samples for Monte Carlo integration}
\STATE $\epsilon$ \COMMENT{Brownian noise intensity}; $\gamma$ \COMMENT{Initial Gauss noise intensity}
\STATE $\text{num\_samples}$ \COMMENT{Number of samples}; $\text{num\_steps}$ \COMMENT{Number of time steps}

\STATE \textbf{Define Functions:}
\STATE $(\sigma_t,\beta_t) \gets \left(\cos^2\left(\frac{\pi t}{2}\right), \sin^2\left(\frac{\pi t}{2}\right)\right)$ \COMMENT{Time-dependent function}
\STATE $\ell_t \gets \epsilon\beta_t + \gamma \sigma_t$ \COMMENT{Noise intensity function}
\STATE $\rho_1(y)$ \COMMENT{Target density function}; $dt\gets1/\text{num\_steps}$ \COMMENT{Step size}

\STATE \textbf{SDE Simulation:}
\STATE $\text{simulate\_sde\_batch}(\text{num\_steps}, \text{num\_samples})$
\STATE \quad $X \gets \text{randn}(D, \text{num\_samples}) \cdot \sqrt{\gamma}$ \COMMENT{Initialize Gauss noise}
\STATE \quad $(\xi_i)_{i=1,\cdots,N} \gets \text{randn}(D, \text{num\_samples})$ \COMMENT{Random samples}
\STATE \quad \textbf{for} $k = 0$ \textbf{to} $\text{num\_steps-1}$ \textbf{do}
\STATE \quad \quad $t \gets k / \text{num\_steps}$ \COMMENT{Current time}
\STATE \quad \quad $\text{drift} \gets \frac{\sum_{i=1}^N(\xi_i  \e^{\beta_t  |X - \sqrt{\ell_t\sigma_t}
\xi_i|^2/(2 \ell_t)+ \left(1 - \beta_t\right) |\xi_i|^2/2} \cdot \rho_1(X - \sqrt{\ell_t\sigma_t}\xi_i))}
{\sum_{i=1}^N(\e^{\beta_t |X - \sqrt{\ell_t\sigma_t} \xi_i|^2/(2 \ell_t)
+ \left(1 - \beta_t\right) |\xi_i|^2/2} \cdot \rho_1(X - \sqrt{\ell_t\sigma_t}\xi_i))}$ \COMMENT{Compute drift}
\STATE \quad \quad $\text{noise} \gets \text{randn\_like}(X)$ \COMMENT{Random noise}
\STATE \quad \quad $X \gets X + \sqrt{\ell_t/\sigma_t}\sigma'_t \cdot \text{drift} \cdot dt + \sqrt{\epsilon\beta_t' dt} \cdot \text{noise}$ \COMMENT{Update samples}
\STATE \quad \textbf{end for}
\STATE \quad \textbf{return} $X$ \COMMENT{Final samples}
\end{algorithmic}
\end{algorithm}
Using the above Algorithm, where the parameters are choosen as follows:
$$
\beta_t=\sin^2\left(\frac{\pi t}{2}\right),\ \ N=100,000,\ \ \eps=1,\ \ \gamma=2,\ \ {\rm num\_steps}=100,
$$
we generated 10,000 sample points and plot the histogram below 
when the dimensions are 2, 7, 10, 15, and  20. For the high-dimensional case $d \geq 3$, we project the sampling points onto the first two coordinates. The figure below presents the sampling results, where the left panel shows the true contour plot of the target distribution, and the right panel displays the histogram of the sampled points for different dimensions.
\begin{figure}[H]
    \centering
    \begin{minipage}[b]{0.3\textwidth}
        \centering
        \includegraphics[width=1.8in, height=1.3in]{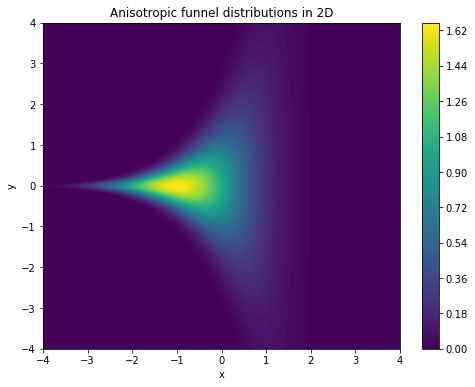}
    \end{minipage}
    \begin{minipage}[b]{0.3\textwidth}
        \centering
        \includegraphics[width=1.8in, height=1.3in]{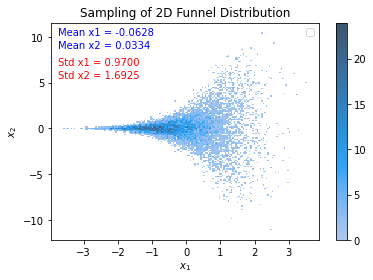}
    \end{minipage}
    \begin{minipage}[b]{0.3\textwidth}
        \centering
        \includegraphics[width=1.8in, height=1.3in]{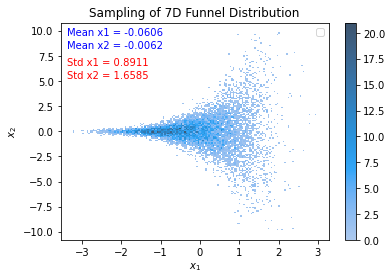}
    \end{minipage}
 
    \begin{minipage}[b]{0.3\textwidth}
        \centering
        \includegraphics[width=1.8in, height=1.3in]{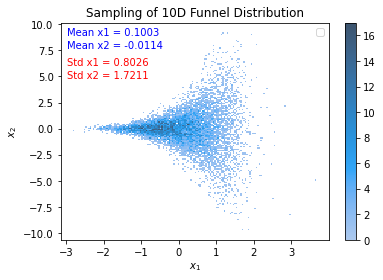}
    \end{minipage}
    \begin{minipage}[b]{0.3\textwidth}
        \centering
        \includegraphics[width=1.8in, height=1.3in]{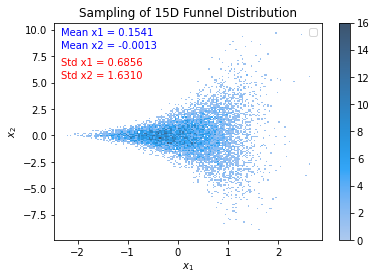}
    \end{minipage}
    \begin{minipage}[b]{0.3\textwidth}
        \centering
        \includegraphics[width=1.8in, height=1.3in]{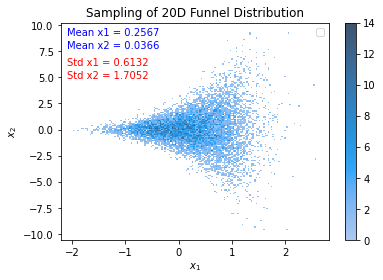}
    \end{minipage}
    
    \caption{Sampling of  high dimensional funnel distribution}
    \label{fig:2D5}
\end{figure}

From the above sampling results, it is evident that as the dimension increases, the performance deteriorates. This is primarily due to the use of Monte Carlo methods to calculate both the numerator and denominator. Since 
$\rho_1(x)$ contains Gaussian distributions, as shown in \cite{Zh24}, 
we can compute a more explicit expression. 
Indeed, by  \eqref{Ex0} and $\rho_1(y)=\phi_1(y_1)\phi_{\e^{2\alpha y_1}}(y_1^*)$, we have
\begin{align*}
b_t(x)&=\frac{\sigma'_t\int_{\mR^d}(x-y) \phi_{\ell_t\sigma_t}(x-\beta_ty)
\phi_1(y_1)\phi_{\e^{2\alpha y_1}}(y_1^*)\dif y}
{\sigma_t\int_{\mR^d}\phi_{\ell_t\sigma_t}(x-\beta_ty)\phi_1(y_1)\phi_{\e^{2\alpha y_1}}(y_1^*)\dif y}.
\end{align*} 
Noting that by Lemma \ref{Le81} below with $\fm=0$ and $\Sigma=\e^{2\alpha y_1}\mI_{d-1}$,
$$
\int_{\mR^{d-1}} \phi_{\ell_t\sigma_t}(x_1^*-\beta_ty_1^*)\phi_{\e^{2\alpha y_1}}(y_1^*)\dif y_1^*
=\phi_{\ell_t\sigma_t+\beta_t^2\e^{2\alpha y_1}}(x_1^*),
$$
and
$$
\int_{\mR^{d-1}}y_1^* \phi_{\ell_t\sigma_t}(x_1^*-\beta_ty_1^*)\phi_{\e^{2\alpha y_1}}(y_1^*)\dif y_1^*
=\frac{\beta_t\e^{2\alpha y_1} x_1^*}{\ell_t\sigma_t+\beta_t^2\e^{2\alpha y_1}}\phi_{\ell_t\sigma_t+\beta_t^2\e^{2\alpha y_1}}(x_1^*),
$$
letting $\xi\sim N(0,1)$, by Fubini's theorem, we have for $x=(x_1,x_1^*)\in\mR\times\mR^{d-1}$,
\begin{align*}
&\int_{\mR^d}(x-y) \phi_{\ell_t\sigma_t}(x-\beta_ty) \phi_1(y_1)\phi_{\e^{2\alpha y_1}}(y_1^*)\dif y\\
&\quad=\int_{\mR}\left(x_1-y_1,x_1^*-\frac{\beta_t\e^{2\alpha y_1} x_1^*}{\ell_t\sigma_t+\beta_t^2\e^{2\alpha y_1}}\right)
\phi_{\ell_t\sigma_t}(x_1-\beta_ty_1) \phi_1(y_1)\phi_{\ell_t\sigma_t+\beta_t^2\e^{2\alpha y_1}}(x_1^*)\dif y_1\\
&\quad=\mE\left[\left(x_1-\xi,\frac{\sigma_t(\ell_t-\beta_t\e^{2\alpha\xi})}{\ell_t\sigma_t+\beta_t^2\e^{2\alpha\xi}}x_1^*\right)
\phi_{\ell_t\sigma_t}(x_1-\beta_t\xi)\phi_{\ell_t\sigma_t+\beta_t^2\e^{2\alpha \xi}}(x_1^*)\right],
\end{align*}
and
$$
\int_{\mR^d}\phi_{\ell_t\sigma_t}(x-\beta_ty)\phi_1(y_1)\phi_{\e^{2\alpha y_1}}(y_1^*)\dif y
=\mE\left[\phi_{\ell_t\sigma_t}(x_1-\beta_t\xi)\phi_{\ell_t\sigma_t+\beta_t^2\e^{2\alpha \xi}}(x_1^*)\right].
$$
Hence,
$$
b_t(x)=\frac{\sigma'_t\mE\left[\left(\frac{x_1-\xi}{\sigma_t},\frac{\ell_t-\beta_t\e^{2\alpha\xi} }{\ell_t\sigma_t+\beta_t^2\e^{2\alpha\xi}}x_1^*\right)\phi_{\ell_t\sigma_t}(x_1-\beta_t\xi)\phi_{\ell_t\sigma_t+\beta_t^2\e^{2\alpha \xi}}(x_1^*)\right]}
{\mE\left[\phi_{\ell_t\sigma_t}(x_1-\beta_t\xi)\phi_{\ell_t\sigma_t+\beta_t^2\e^{2\alpha \xi}}(x_1^*)\right]}.
$$
Now we consider the following SDE
$$
\dif X_t=b_t(X_t)\dif t+\sqrt{\eps\beta'_t}\dif W_t,\ \ X_0\sim N(0,\mI_d).
$$
Using this SDE, one can devise a similar algorithm for sampling from $f$. Here are the simulated results.
Since this expression only contains the integral with respect to the first variable, it appears to work well for high-dimensional cases.
\begin{figure}[H]
    \centering
    \begin{minipage}[b]{0.3\textwidth}
        \centering
        \includegraphics[width=1.8in, height=1.3in]{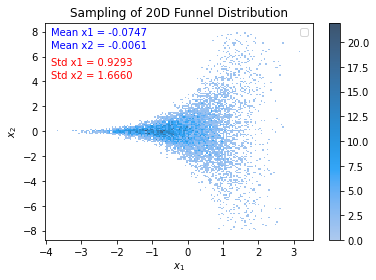}
    \end{minipage}
    \begin{minipage}[b]{0.3\textwidth}
        \centering
        \includegraphics[width=1.8in, height=1.3in]{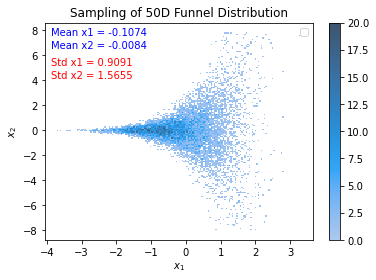}
    \end{minipage}
    \begin{minipage}[b]{0.3\textwidth}
        \centering
        \includegraphics[width=1.8in, height=1.3in]{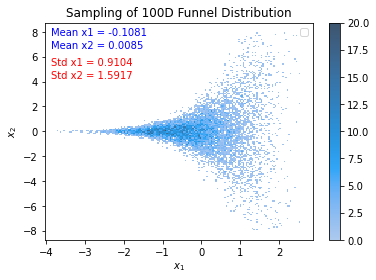}
    \end{minipage}
    \caption{Sampling  of 20D, 50D, 100D-funnel distributions}
\end{figure}

\subsection{Sampling Gaussian Mixture Distribution}
Let $n\in\mN$ be the number of Gaussian distributions.
We consider the following Gaussian  mixture distribution:
$$
\rho_1(x)=\sum_{i=1}^nw_i\sN(x|\fm_i,\Sigma_i),
$$
where $\fm_i\in\mR^d$ is the mean value and $\Sigma_i$ is the $d$-dimensional covariance matrix, and
$w_i\in(0,1)$ are the weights with $w_1+w_2+\cdots+w_n=1$.
Consider the zero order scheme in \eqref{AA36}
$$
b_t(x)=\frac{\sigma_t'\int_{\mR^d}(x-y)\e^{-|x-\beta_t y|^2/(2\ell_t\sigma_t)} \rho_1(y)\dif y}
{\sigma_t\int_{\mR^d}\e^{-|x-\beta_t y|^2/(2\ell_t\sigma_t)} \rho_1(y)\dif y}
=\frac{\sigma_t'\sum_iw_i\mE[(x-\xi_i)\e^{-|x-\beta_t \xi_i|^2/(2\ell_t\sigma_t)}]}
{\sigma_t\sum_iw_i\mE\e^{-|x-\beta_t \xi_i|^2/(2\ell_t\sigma_t)}},
$$
where $\ell_t:=\eps\beta_t+\gamma\sigma_t$ with $\eps,\gamma>0$ and
$$
\xi_i\sim N(\fm_i,\Sigma_i). 
$$
By Lemma \ref{Le81}, we can write
\begin{align*}
b_t(x)&=\frac{\sigma'_t\sum_i w_i\sN(x|\beta_t\fm_i, A_i)v_i(x)}
{\sum_iw_i\sN(x|\beta_t\fm_i, A_i)}
=\frac{\sigma'_t\sum_i w_i\e^{(-\log\det(A_i)-\<A^{-1}_i(x-\beta_t\fm_i), x-\beta_t\fm_i\>)/2}v_i(x)}
{\sum_iw_i\e^{(-\log\det(A_i)-\<A^{-1}_i(x-\beta_t\fm_i), x-\beta_t\fm_i\>)/2}},
\end{align*}
where $A_i=\ell_t\sigma_t \mI_d+\beta_t^2\Sigma_i$ and
$$
v_i(x):=(x-A^{-1}_i(\beta_t\Sigma_ix+\ell_t\sigma_t\fm_i))/\sigma_t
=A^{-1}_i[(\ell_t\mI_d-\beta_t\Sigma_i)x-\ell_t\fm_i].
$$
Thus, we can sample the Gaussian mixture using the following SDE and its Euler scheme:
$$
\dif X_t=\frac{\sigma'_t\sum_i w_i\sN(X_t|\beta_t\fm_i, A_i)v_i(X_t)}
{\sum_iw_i\sN(X_t|\beta_t\fm_i, A_i)}\dif t+\sqrt{\eps\beta'_t}\dif W_t,\ \ X_0\sim N(0,\gamma\mI_d).
$$
\noindent
{\bf Example 1: $n=1$.} Let $\eta\sim N(\fm,\Sigma)$. In this case, from the above expression we have
$$
b_t(x)=\sigma_t'( \ell_t\sigma_t\mI_d+\beta^2_t\Sigma)^{-1}
\left[(\ell_t\mI_d-\beta_t\Sigma)x-\ell_t \fm\right].
$$
Now we consider the following linear vector-valued SDE
$$
\dif X_t=\sigma_t'( \ell_t\sigma_t\mI_d+\beta^2_t\Sigma)^{-1}
\left[(\ell_t\mI_d-\beta_t\Sigma)X_t-\ell_t \fm\right]\dif t+\sqrt{\eps\beta'_t}\dif W_t,\ \ X_0\sim N(0,\gamma\mI_d).
$$
Using this SDE, we can sample from a normal distribution even when the covariance matrix $\Sigma$ is potentially degenerate.  
We set $d = 100$, $\fm = 0$, and let $\Sigma$ be a degenerate covariance matrix. We generated 2,000 sample points and plotted their scatter plots along the $(0, 1)$-dimension and $(23, 54)$-dimension, as shown below.
In the plots, the red circular line represents the true Gaussian contour, and the yellow areas indicate the covariance matrix for the corresponding two dimensions.
\begin{figure}[H]
    \centering
    \begin{minipage}[b]{0.42\textwidth}
        \centering
        \includegraphics[width=2.5in, height=1.3in]{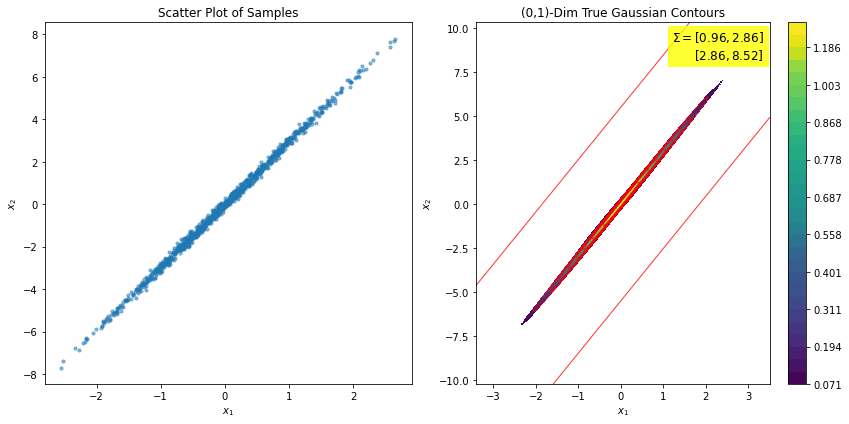}
    \end{minipage}
    \begin{minipage}[b]{0.42\textwidth}
        \centering
        \includegraphics[width=2.5in, height=1.3in]{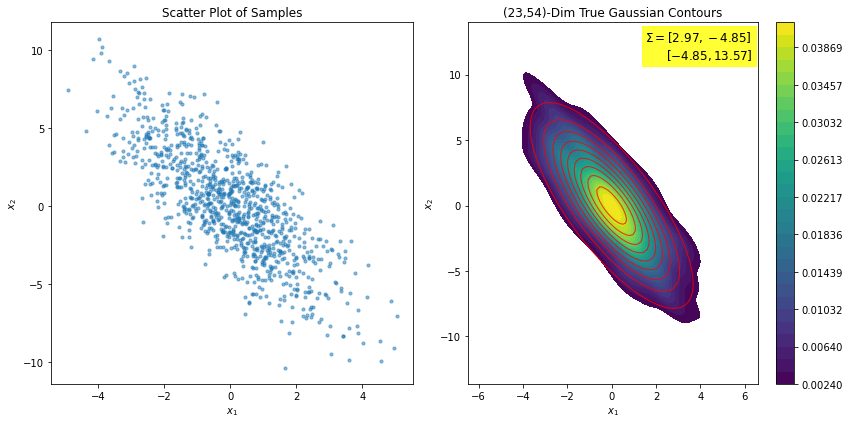}
    \end{minipage}
    \caption{Sampling of Gaussian distribution with degenerate covariance matrix}
\end{figure}

\noindent{\bf Example 2: $n>1$ and $\Sigma_i={\rm Diag}(\alpha_{i1},\cdots,\alpha_{id})$.}  
In this case for $i=1,\cdots,n$, we have 
$$
A_i={\rm Diag}\left(\ell_t\sigma_t+\beta_t^2\alpha_{i1},\cdots,\ell_t\sigma_t+\beta_t^2\alpha_{id}\right),\ \ 
\det(A_i)=\Pi_{j=1}^d(\ell_t\sigma_t+\beta^2_t\alpha_{ij}),
$$
and
\begin{align*}
v_i(x)=\left(\tfrac{(\ell_t-\beta_t\alpha_{i1})x_1-\ell_t\fm_{i1}}{\ell_t\sigma_t+\beta_t^2\alpha_{i1}},\cdots,
\tfrac{(\ell_t-\beta_t\alpha_{id})x_d-\ell_t\fm_{id}}{\ell_t\sigma_t+\beta_t^2\alpha_{id}}\right).
\end{align*}
Thus,
$$
\<A^{-1}_i(x-\beta_t\fm_i), x-\beta_t\fm_i\>
=\sum_{j=1}^d\frac{(x_j-\beta_t\fm_{ij})^2}{\ell_t\sigma_t+\beta_t^2\alpha_{ij}}.
$$
Hence,
\begin{align*}
b_t(x)=\frac{\sigma_t'\sum_iw_i\exp\left\{-\sum_{j=1}^d\left(\frac{(x_j-\beta_t\fm_{ij})^2}{2(\ell_t\sigma_t+\beta_t^2\alpha_{ij})}
+\frac{\log(\ell_t\sigma_t+\beta_t^2\alpha_{ij}))}2\right)\right\}v_i(x)}
{\sum_iw_i\exp\left\{-\sum_{j=1}^d\left(\frac{(x_j-\beta_t\fm_{ij})^2}{2(\ell_t\sigma_t+\beta_t^2\alpha_{ij})}+\frac{\log(\ell_t\sigma_t+\beta_t^2\alpha_{ij})}2\right)\right\}}.
\end{align*}
Using the above drift $b$, we can sample from a Gaussian mixture distribution.  
We set $d =n= 100$, $\fm = 0$, and let $\Sigma_i$ be a diagonal matrix with $\alpha_{ij}\in(0,0.005)$. 
We generated 2,000 sample points and plotted their scatter plots along the $(20, 10)$-dimension and $(3, 0)$-dimension, 
as shown below.
\begin{figure}[H]
    \centering
    \begin{minipage}[b]{0.42\textwidth}
        \centering
        \includegraphics[width=2in, height=1.3in]{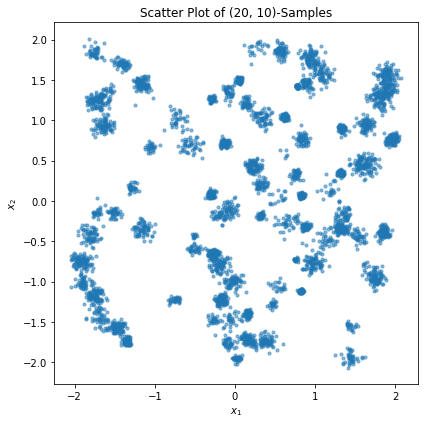}
    \end{minipage}
    \begin{minipage}[b]{0.42\textwidth}
        \centering
        \includegraphics[width=2in, height=1.3in]{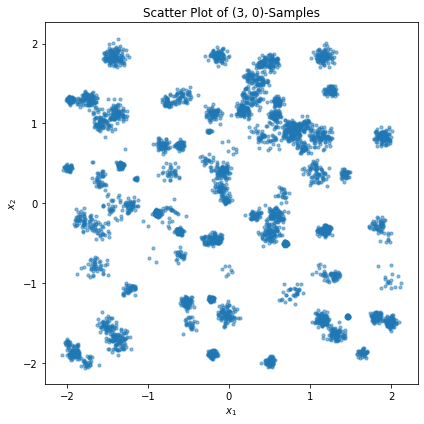}
    \end{minipage}
    \caption{Sampling of 100D Gaussian mixture distribution with 100-components}
\end{figure}

\section{Appendix}

We need the following simple lemma.

\bl\label{Le81}
Let $\fm \in \mathbb{R}^d$ and $\Sigma$ be a $d \times d$-symmetric and non-negative definite matrix. Suppose that $\xi \sim N(\fm, \Sigma)$. Then for any $\beta, \sigma > 0$ and $a \in \mathbb{R}, b \in \mathbb{R}^d$, we have
$$
\mathbb{E}\left[(a\xi + b) \phi_\sigma(x - \beta \xi)\right]= \sN(x|\beta\fm,A)\left[a A^{-1} (\beta \Sigma x + \sigma \fm) + b\right],
$$
where $\sN(x|\fm,\Sigma)$ is defined by \eqref{Nor1}, and
$$
\phi_\sigma(x) := (2\pi\sigma)^{-d/2}\e^{-|x|^2 / (2\sigma)}, \quad A := \sigma \mathbb{I}_d + \beta^2 \Sigma.
$$
\el

\begin{proof}
Without loss of generality, we may assume $\det(\Sigma) > 0$. Otherwise, we may consider $\Sigma_\eps = \Sigma + \eps \mathbb{I}_d$. By definition and the change of variable, we have
\begin{align*}
\mathbb{E}\left[(a\xi + b) \phi_\sigma(x - \beta \xi)\right]
&= \int_{\mathbb{R}^d} (a y + b) \frac{\e^{-|x - \beta y|^2 / (2\sigma)}\e^{-\langle \Sigma^{-1}(y - \fm), y - \fm \rangle / 2}}{(2\pi)^d \sqrt{\sigma^d \det(\Sigma)}} \dif y \\
&= \int_{\mathbb{R}^d} (a y + a \fm + b) \frac{\e^{-|x - \beta(y + \fm)|^2 / (2\sigma) - \langle \Sigma^{-1} y, y \rangle / 2}}{(2\pi)^d \sqrt{\sigma^d \det(\Sigma)}} \dif y.
\end{align*}
Note that
\begin{align*}
|x - \beta(y + \fm)|^2 + \sigma \langle \Sigma^{-1} y, y \rangle
&= |x - \beta \fm|^2 + 2 \beta \langle x - \beta \fm, y \rangle + \beta^2 |y|^2 + \sigma \langle \Sigma^{-1} y, y \rangle \\
&= |x - \beta \fm|^2 + 2 \beta \langle x - \beta \fm, y \rangle + \langle B y, y \rangle \\
&= |x - \beta \fm|^2 + \langle B(y - v), (y - v) \rangle - \langle B v, v \rangle,
\end{align*}
where
$$
B := \sigma \Sigma^{-1} + \beta^2 \mathbb{I}_d, \quad v = \beta B^{-1}(x - \beta \fm).
$$
Hence,
\begin{align*}
\mathbb{E}\left[(a\xi + b)\e^{-|x - \beta \xi|^2 / (2\sigma)}\right]
&= \frac{\e^{(\langle B v, v \rangle - |x - \beta \fm|^2) / (2\sigma)}}{(2\pi)^d \sqrt{\sigma^d \det(\Sigma)}}
\int_{\mathbb{R}^d} (a y + a \fm + b)\e^{-\langle B(y - v), (y - v) \rangle / (2\sigma)} \dif y \\
&= \frac{\e^{(\langle B v, v \rangle - |x - \beta \fm|^2) / (2\sigma)}}{(2\pi)^d \sqrt{\sigma^d \det(\Sigma)}}
(a v + a \fm + b) \int_{\mathbb{R}^d}\e^{-\langle B y, y \rangle / (2\sigma)} \dif y \\
&= \frac{\e^{(\langle B v, v \rangle - |x - \beta \fm|^2) / (2\sigma)}}{\sqrt{(2\pi)^d \det(\Sigma) \det(B)}} (a v + a \fm + b).
\end{align*}
Finally, noting that
$$
v = \beta (\sigma \Sigma^{-1} + \beta^2 \mathbb{I}_d)^{-1} (x - \beta \fm)
= \beta (\sigma \mathbb{I}_d + \beta^2 \Sigma)^{-1} \Sigma (x - \beta \fm),
$$
we have
\begin{align*}
\langle B v, v \rangle - |x - \beta \fm|^2
&= \langle x - \beta \fm, [\beta^2 (\sigma \mathbb{I}_d + \beta^2 \Sigma)^{-1} \Sigma - \mathbb{I}_d] (x - \beta \fm) \rangle \\
&= -\sigma \langle x - \beta \fm, (\sigma \mathbb{I}_d + \beta^2 \Sigma)^{-1} (x - \beta \fm) \rangle
\end{align*}
and
$$
v + \fm = \beta (\sigma \mathbb{I}_d + \beta^2 \Sigma)^{-1} \Sigma (x - \beta \fm) + \fm
= (\sigma \mathbb{I}_d + \beta^2 \Sigma)^{-1} (\beta \Sigma x + \sigma \fm).
$$
The desired formula now follows.
\end{proof}

We need the following well-known results. For the readers' convenience, we provide a detailed proof.

\bl\label{Le82}
Let $\bxi^N := (\xi_1, \dots, \xi_N)$ be a sequence of i.i.d. random variables in $\mathbb{R}^d$ with common distribution $N(0, \mathbb{I}_d)$. For any $p \geq 1$, there is a constant $C_{p,d} > 0$ such that
$$
\left\||\bxi|_*^N\right\|_p \leq C_{p,d} \sqrt{\ln N},
$$
where $|\bxi|_*^N := \sup_{i=1,\dots,N} |\xi_i|$.
\el

\begin{proof}
First of all, for $t > 0$, we have
$$
\mathbb{P}(|\bxi|_*^N \leq t) = \mathbb{P}\left(\sup_{i=1,\dots,N} |\xi_i| \leq t\right) = \mathbb{P}\left(|\xi_1| \leq t\right)^N.
$$
Hence,
$$
\||\bxi|_*^N\|^p_p = \int_0^\infty \mathbb{P}(|\bxi|_*^N > t) \dif t^p = \int_0^\infty \left(1 - \mathbb{P}\left(|\xi_1| \leq t\right)^N\right) \dif t^p.
$$
On the other hand, using the Chernoff bound for chi-squared random variables, we know
\begin{align*}
\mathbb{P}(|\xi_1| > t) 
&\leq \left(\inf_{s > 0}\e^{-s t^2} \mathbb{E}\e^{s |\xi|_1^2}\right) \wedge 1 \\
&= \left(\inf_{s \in (0, \frac{1}{2})} (\e^{-s t^2} (1 - 2s)^{-d})\right) \wedge 1 \\
&\leq \exp\left(-\frac{t^2 - d}{2} - \frac{d}{2} \log\left(\frac{t^2}{d}\right)\right), \quad t^2 > d.
\end{align*}
For $t^2 > 4d$, it is easy to see that
$$
\inf_{s \in (0, \frac{1}{2})} (\e^{-s t^2} (1 - 2s)^{-d}) = \left(\frac{t^2}{2d}\right)^d\e^{-t^2 / 2 + d}.
$$
In particular, for $t^2 > 4d$ and some $c_0 > 0$,
$$
\mathbb{P}(|\xi_1| > t) \leq \left(\frac{t^2}{2d}\right)^d\e^{-t^2 / 2 + d} \leq\e^{-c_0 t^2}, \quad t^2 > 4d.
$$
Hence,
$$
\||\bxi|_*^N\|^p_p \leq \int_0^{2 \sqrt{d}} \dif t^p + \int^\infty_{2 \sqrt{d}} \left(1 - \left(1 -\e^{-c_0 t^2}\right)^N\right) \dif t^p \leq (2 \sqrt{d})^p + \sI,
$$
where
$$
\sI := \int^\infty_0 \left(1 - \left(1 -\e^{-c_0 t^2}\right)^N\right) \dif t^p.
$$
For $\sI$, by the change of variable $1 -\e^{-c_0 t^2} = x$, we have
\begin{align*}
\sI &= \int^1_0 \left(1 - x^N\right) \dif \left(\frac{\ln(1 - x)^{-1}}{c_0}\right)^{p/2} \\
&= N \int^1_0 \left(\frac{\ln(1 - x)^{-1}}{c_0}\right)^{p/2} x^{N-1} \dif x =: \sI_1 + \sI_2,
\end{align*}
where
$$
\sI_1 := N \int^1_{1/N} \left(\frac{\ln x^{-1}}{c_0}\right)^{p/2} (1 - x)^{N-1} \dif x
$$
and
$$
\sI_2 := N \int^{1/N}_0 \left(\frac{\ln x^{-1}}{c_0}\right)^{p/2} (1 - x)^{N-1} \dif x.
$$
For $\sI_1$, we have
$$
\sI_1 \leq \left(\frac{\ln N}{c_0}\right)^{p/2} N \int^1_{1/N} (1 - x)^{N-1} \dif x
= \left(\frac{\ln N}{c_0}\right)^{p/2} (1 - 1/N)^N \leq \frac{(\ln N)^{p/2}}{c_0^{p/2} \e}.
$$
For $\sI_2$, we have
\begin{align*}
\sI_2 &\leq N c_0^{-p/2} \int^{1/N}_0 \left(\ln x^{-1}\right)^{p/2} \dif x = N c_0^{-p/2} \int^\infty_{\ln N} t^{p/2}\e^{-t} \dif t \\
&= N c_0^{-p/2} \left[(\ln N)^{p/2}\e^{-\ln N} + \frac{p}{2} \int^\infty_{\ln N} t^{p/2 - 1}\e^{-t} \dif t\right] \\
&= c_0^{-p/2} (\ln N)^{p/2} + N c_0^{-p/2} \frac{p}{2} (\ln N)^{p/2 - 1} + \cdots \leq C (\ln N)^{p/2},
\end{align*}
where $C=C_{p,d}>0$.
Combining the above calculations, we obtain the desired estimate.
\end{proof}

We consider the following SDE:
\begin{align}\label{SDE9}
\dif X^i_t = b^i_t(X^i_t) \dif t + \sqrt{2 a_t} \dif W_t, \quad i = 1, 2,
\end{align}
where $a_t > 0$ and $b^i: \mathbb{R}_+ \times \mathbb{R}^d \to \mathbb{R}^d$ are two Borel measurable functions. We recall the following entropy formula for the solutions of classical SDEs, which is a consequence of Girsanov's theorem (see \cite[Lemma 4.4 and Remark 4.5]{La21} for the most general form).

\bl\label{Lem34}
Suppose that for each $i = 1, 2$, there is a unique solution to SDE \eqref{SDE9}. Let $\mu^i_t = \mathbb{P} \circ (X^i_t)^{-1}$ be the time marginal distribution. Then
$$
\cH(\mu^1_t | \mu^2_t) = \cH(\mu^1_0 | \mu^2_0) + \mathbb{E}\left(\int^t_0 a^{-1}_s |b^1_s(X^1_s) - b^2_s(X_s^1)|^2 \dif s\right),
$$
where $\cH(\mu^1_t|\mu^2_t)$ is the relative entropy defined by \eqref{Rel1}.
\el

\vspace{5mm}

{\bf Acknowledgement:} I would like to thank Zimo Hao, Haojie Hou, Qi Meng, Zhenyao Sun and  Rongchan Zhu
for their valuable discussions.

\end{document}